%% file: AppM2_MultiD.tex
\documentclass[10pt]{article}

\usepackage{pdfsync}
\usepackage{times,amsbsy,amsmath,amsthm,amssymb,amscd,amsfonts,mathrsfs}
\usepackage{graphicx,epstopdf}
\usepackage[hang]{subfigure}
\usepackage{color}
\usepackage{macros}
\usepackage{setspace}
\usepackage{sansmath}
\usepackage{extarrows}
\usepackage{tikz}
\usetikzlibrary{shadings,intersections}
\usetikzlibrary{calc}

\newcommand\bOmega{\boldsymbol{\Omega}}

\newcommand{\pd}[2]{\dfrac{\partial #1}{\partial #2}}
\newcommand{\dd}{\,{\rm d}}
\newcommand\Beta{\boldsymbol{B}}
\newcommand\Trace{\mathrm{Trace}}
\newcommand\bbr{\boldsymbol{\mathit{r}}}
\newcommand\bbrE{\boldsymbol{\mathit{E}}}
\newcommand\bbrR{\boldsymbol{\mathit{R}}}

\newcommand\bbrT{\boldsymbol{\mathit{T}}}
\newcommand\bbrV{\boldsymbol{\mathit{V}}}
\newcommand\Identity{\boldsymbol{\mathit{I}}}
\def\det{\mathit{det}}

\newtheorem{remark}{Remark}[section]
\newtheorem{theorem}{Theorem}[section]
\newtheorem{lemma}{Lemma}[section]

\newtheorem{corollary}{Corollary}[section]

\newcommand\mansatz{\hat I_M}
\newcommand\bansatz{\hat I_B}

\begin{document}

\title{ 3D $B_2$ Model for Radiative Transfer Equation \\ Part I:
  Modelling }

\author{Ruo Li\thanks{School of Mathematical Sciences, Peking
    University and State Key Laboratory of Space Weather, Chinese
    Academy of Sciences, Beijing. Email: \tt{rli@math.pku.edu.cn}}
  ~and~ Weiming Li\thanks{School of Mathematical Sciences, Peking
    University. Email: \tt{liweiming@pku.edu.cn}}}

\date{\today}

\maketitle

\input{abstract}

\input{introduction}
\input{preliminary}
\input{model}
\input{analysis}
\input{conclusion}
\input{acknowledgements}

\bibliographystyle{plain}
\bibliography{../references}

\end{document}

%% file: abstract.tex

\begin{abstract}
  We extend to three-dimensional space the approximate $M_2$ model for
  the slab geometry studied in \cite{alldredge2016approximating}. The
  $B_2$ model therein, as a special case of the second order extended
  quadrature method of moments (EQMOM), is proved to be globally
  hyperbolic. The model we propose here extends EQMOM to multiple
  dimensions following the idea to approximate the maximum entropy
  closure for the slab geometry case. Like the $M_2$ closure, the
  ansatz of the new model has the capacity to capture both isotropic
  and beam-like solutions. Also, the new model has fluxes in
  closed-form; thus, it is applicable to practical numerical
  simulations. The rotational invariance, realizability, and
  hyperbolicity of the model are studied.
\end{abstract}



%% file: introduction.tex

\section{Introduction}

The radiative transfer equations describe the transportation of light in a medium
\cite{pomraning1973equations}. They are \emph{kinetic equations}, and the unknown 
is the specific intensity of photons. The specific intensity is a function of time, 
spatial coordinates, frequency, and angular variables. The moment method is an 
efficient approach for reducing the computation cost brought about by the high-dimensionality of 
variables of kinetic equations.

Moments are obtained by integrating the specific intensity against
monomials of the angular variables. In many applications, the
quantities of interest are the few lowest order moments. Therefore
moments are good choices for discretizing the angular variables.
However, moment systems are not closed. Closing the system by
specifying a constitutive relationship is known as the
\emph{moment-closure problem}. One approach towards moment-closure is
to recover the angular dependence of the specific intensity from the
known moments. The reconstructed specific intensity is called an
\emph{ansatz}. Ideally, the ansatz should be non-negative for all
moments which can be generated by a non-negative distribution. Also,
one would like the system to be hyperbolic since hyperbolicity is
necessary for the local well-posedness of Cauchy problem.  Other
natural requirements include that the ansatz satisfies rotational
invariance and reproduces the isotropic distribution at equilibrium.
Numerous forms of ans\"atze have been studied in the literature
\cite{pomraning1973equations, Lewis-Miller-1984}.  Yet, in
multi-dimensional cases, the maximum entropy method, referred to as
the $M_n$ model, is perhaps the
only method known so far to have both realizability and global
hyperbolicity \cite{Dubroca-Feugas-1999}. However, the flux functions
of the maximum entropy method are generally not explicit
\footnote{
With the first order maximum entropy model for the grey equations as the only exception
\cite{Dubroca-Feugas-1999}.
}, so
numerically computing such models involve solving highly
nonlinear and probably ill-conditioned optimization problems
frequently. There have been continuous efforts on speeding up the
computation process \cite{alldredge2012high, alldredge2014adaptive,
  garrett2015optimization}. Recently, there are also attempts in
deriving closed-form approximations of the maximum entropy closure in
order to avoid the expensive computations. For 1D cases, an
approximation to the $M_n$ models using the Kershaw closure is given
in \cite{schneider2016kershaw}.  For multi-dimensional cases, a model
based on directly approximating the closure relations of the $M_1$
and $M_2$ methods is proposed in \cite{pichard2016approximation}.  Our
work in this paper also aims at constructing closed-form
approximations of the maximum entropy model. Like
\cite{pichard2016approximation}, we seek a closed-form approximation
to the $M_2$ method in 3D. But unlike \cite{pichard2016approximation},
we derive our model from an ansatz with some similarity to that of the
$M_2$ model.

In a previous study \cite{alldredge2016approximating}, we analyzed the
second order extended quadrature method of moments
(EQMOM) introduced in \cite{vikas2013radiation} which we call the $B_2$
model.  
In this work, we propose an approximation of the $M_2$ model in
3D space by extending the $B_2$ model studied in
\cite{alldredge2016approximating} to 3D. The reason for this approach
is that the $B_2$ ansatz shares the following properties with the $M_2$ ansatz:
\begin{enumerate}
\item it interpolates smoothly between isotropic and Dirac
  distribution functions;
\item it captures anisotropy in opposite directions. 
\end{enumerate}
The $B_2$ closure in \cite{alldredge2016approximating} is for slab
geometries. Preserving rotational invariance when extending it to 3D
space is non-trivial. We use the sum of the axisymmetric $B_2$
ans\"atze in three mutually orthogonal directions as the ansatz for a
second order moment model in 3D space. This new model is referred to
as the {\it 3D $B_2$ model}. The consistency of known moments requires
the three mutually orthogonal directions to be the three eigenvectors
of the second-order moment matrix. We point out that there are three free
parameters in the ansatz of the 3D $B_2$ model after the consistency of 
known moments is fulfilled. These parameters are
specified as functions of the first-order moments and the eigenvalues
of the second-order moment matrix. We prove that the 3D $B_2$ model is
rotationally invariant.  The region where the model possesses a
non-negative ansatz is illustrated, as well as the hyperbolicity region of the model
with vanished first-order moment. Though far from
perfect, the 3D $B_2$ model shares some important features of the
$M_2$ closure. Also, the model has explicit flux functions, making it
very convenient for numerical simulations.

The rest of this paper is organized as follows. In Section 2 we recall
the basics of moment models, and briefly, introduce the $M_2$ method
as well as the $B_2$ model for 1D slab geometry. In Section 3 we
propose the 3D $B_2$ model. In Section 4 we analyze its
properties. Finally, in Section 5 we summarize and discuss future
work.


%% file: preliminary.tex

\section{Preliminaries}\label{sec:preliminary}
The specific intensity $I(t,\br,\nu,\bOmega)$ is governed by the
radiative transfer equation
\begin{equation}\label{eq:rt}
  \dfrac{1}{c}\pd{I}{t}+\bOmega\cdot\nabla I
  = \cC(I),   
\end{equation}
where $c$ is the speed of light. The variables in the equation are
time $t \in \mathbb{R}^+$, the spatial coordinates
$\br = (x, y, z)\in \mathbb{R}^3$, the angular variables $\bOmega = \left(
\Omega_x, \Omega_y, \Omega_z \right) \in \bbS^2$,
and frequency $\nu\in\mathbb{R}^+$. The right-hand side
$\cC(I)$ describes the interactions between photons and the background medium
and are not the focus of this paper. A typical right-hand side takes the form 
\[
  \cC(I) = -\sigma_a I(\bsOmega) - \sigma_s
  \left(I(\bsOmega) - \frac1{4\pi}\int_{\bbS^2} I(\bsOmega)
  \dd\bsOmega \right),
\]
where $\sigma_a$ and $\sigma_s$ are constant parameters. 
We introduce the moment method in the context of second order models. Let
\begin{equation}\label{eq:v-def}
\begin{array}{rllll}
  \bv = [ &1, &&& \\ [2mm]
  & (\bOmega\cdot\be_x), & (\bOmega\cdot\be_y), & (\bOmega\cdot\be_z),
  &\\ [2mm]
  & (\bOmega\cdot\be_x)^2,&(\bOmega\cdot\be_x)(\bOmega\cdot\be_y),
  &(\bOmega\cdot\be_x)(\bOmega\cdot\be_z), & \\
  && (\bOmega\cdot\be_y)^2,&(\bOmega\cdot\be_y)(\bOmega\cdot\be_z) &]^T.
\end{array}
\end{equation}
Use $\be_x$, $\be_y$ and $\be_z$ to denote the unit vectors along the
coordinate axes. Define
$$
\Vint{ \psi } := \int_{\bbS^2} \psi(\nu, \bOmega)\, \dd\bOmega.
$$
Multiplying equation \eqref{eq:rt} by the vector $\bv$ defined in \eqref{eq:v-def}
and integrating over the angular variables give
\begin{equation}\label{eq:moment-eq}
  \dfrac{1}{c}\pd{\left\langle \bv I\right\rangle}{t}
  + \pd{\Vint{\Omega_x \bv I}}{x} + \pd{\Vint{\Omega_y \bv I}}{y}
  + \pd{\Vint{\Omega_z \bv I}}{z}
  = \left\langle\bv \mathcal{C}(I)\right\rangle.
\end{equation}
In system \eqref{eq:moment-eq}, the time evolution of second-order moments
relies on third-order moments. Therefore \eqref{eq:moment-eq}
is not a closed system. If we approximate the third-order moments in \eqref{eq:moment-eq}
using lower order moments, we could get a closed system. 
Let\footnote{
The notation $a \simeq b$ means `$a$ is an 
approximation of $b$.'
}
\[
  E^0 \simeq \Vint{I}, \quad
  \bE^1 \simeq \Vint{\bOmega I}, \quad
  \bbrE^2 \simeq \Vint{\bOmega\otimes\bOmega I}, \quad
  \bbrE^3 \simeq \Vint{\bOmega\otimes\bOmega\otimes\bOmega
  I}.
\]
A closed system of equations has the form
\begin{equation}\label{eq:moment-closure}
  \begin{split}
    & \dfrac1c\pd{E^0}{t} + \nabla\cdot\bE^1 = r^0(E^0, \bE^1, \bbrE^2), \\
    & \dfrac1c \pd{\bE^1}{t} + \nabla \cdot \bbrE^2 = 
    \br^1(E^0, \bE^1, \bbrE^2), \\
    & \dfrac1c \pd{\bbrE^2}{t} + \nabla \cdot \left[\bbrE^3(E^0, \bE^1,
  \bbrE^2)\right] = 
    \bbr^2(E^0, \bE^1, \bbrE^2).
  \end{split}
\end{equation}
The choice of $\bbrE^3$, $r^0$, $\br^1$, and $\bbr^2$ specify a
\emph{closure}.  The system \eqref{eq:moment-closure} is a second
order moment model.  The following properties of a moment model
concern us the most, which were frequently discussed in the
literature.
\begin{itemize}
\item[] {\bf Rotational invariance: }
Consider a conservation law in multi-dimensions,
\begin{equation}\label{eq:csv}
  \pd{\bU}{t} + \sum\limits_{i=1}^D 
  \pd{\bF_i(\bU)}{x_i} = 0.
\end{equation}
It satisfies rotational invariance if for any unit
vector $\bn=(n_1, \cdots, n_D)^T \in\bbR^D$, there exists a matrix
$\mathbb{T}$ depending on $\bn$, such that
\[
  \sum\limits_{i=1}^D n_i \bF_i(\bU) = 
  \mathbb{T}^{-1} \bF_1(\mathbb{T} \bU).
\]

\item[] {\bf Hyperbolicity:}
Let $\boldsymbol{\mathrm{J}}_i$ be the Jacobian matrix
of the flux function $\bF_i$ in equation \eqref{eq:csv}.
The system \eqref{eq:csv} is hyperbolic if for any
unit vector $\bn\in\bbR^D$, 
$\sum\limits_{i=1}^D n_i \boldsymbol{\mathrm{J}}_i$
is real diagonalizable.

\item[] {\bf Realizability:}
The realizability domain is defined as moments which 
could be generated by a nonnegative distribution
function \cite{junk2000maximum}. A closure is said
to be realizable if the higher order moments it
closes belong to the realizability domain.

For one dimensional problem, \cite{CurFial91} gives necessary and
sufficient conditions for realizability. Its results cover moments of
arbitrary order.  For multi-dimensional case, only the conditions for
the first and second order models are currently known
\cite{kershaw1976flux}, while the conditions for moments of higher
order remain open problems.
\end{itemize}

The maximum entropy models are equipped with all the properties
mentioned above.  For detailed discussions we refer to
\cite{junk2000maximum, levermore1996moment,Dubroca-Feugas-1999}.  We
review the principles for deriving the maximum entropy models by
taking the second order case as an example. It is called the
$M_2$ model.  Solve the following constrained variational minimization
problem
\begin{equation}
\label{eq:mn-opt}
\begin{split}
  \minimize \:\:  & H(I)\\
 \st \:\: & \Vint{ I }= E^0, \Vint{\bOmega I} = \bE^1, 
 \text{and } \Vint{\bOmega \otimes \bOmega I} = \bbrE^2
\end{split}
\end{equation}
where $H(I)$ is the Bose-Einstein entropy
\begin{equation}\label{eq:be-entropy}
 H(I) := \Vint{\frac{2 k_B \nu^2}{c^3} \left(
  \chi I \log \left( \chi I \right)
  - \left(\chi I + 1 \right) \log \left( \chi I + 1 \right)
  \right)},
\end{equation}
where $\chi = \dfrac{c^2}{2 \hbar \nu^3}$.  This gives us an ansatz
\begin{equation}\label{eq:be-ansatz}
 \mansatz(\nu, \bOmega) =
  \dfrac{1}{\chi} \left( \exp \left( -\dfrac{\hbar \nu}{k_B}
    \bsalpha\cdot\bv \right)-1 \right)^{-1},
\end{equation}
where $\bsalpha \cdot \bv $ is a second order polynomial of
$\bOmega \in \bbS^2$.  The parameters $\bsalpha$ is the unique vector
such that
$$\Vint{ \mansatz} = E^0, \quad 
\Vint{\bOmega \mansatz} = \bE^1,
\text{and} \quad 
\Vint{\bOmega \otimes \bOmega \mansatz} = \bbrE^2.
$$
The $M_2$ method is defined by taking
\begin{equation}\label{eq:mn-flux-and-production}
  \begin{array}{ll}
    \bbrE^3 := \Vint{\bOmega \otimes \bOmega \otimes \bOmega \mansatz},
    \quad & r^0 := \Vint{ \cC (\mansatz)},\\
    \br^1 := \Vint{ \bOmega \cC (\mansatz)}, \quad
    & \bbr^2 := \Vint{ \bOmega \otimes \bOmega \cC
    (\mansatz)}.
  \end{array}
\end{equation}
in \eqref{eq:moment-closure}.

However, the $M_2$ closure is not given explicitly, so
\eqref{eq:mn-flux-and-production} has to be computed by solving the
optimization problem \eqref{eq:mn-opt} numerically. The numerical
optimization at each time step for all spatial grid is extremely 
expensive.

Recent work \cite{pichard2016approximation} proposes an approximation
of the $M_2$ method in multi-dimensions by directly approximating its
closure relation, though the corresponding ansatz to the closure is
not clarified. We adopt the approach of constructing an ansatz to
approximate the $M_2$ ansatz, then the closure relation is given
naturally as in \eqref{eq:mn-flux-and-production}.

In a previous work \cite{alldredge2016approximating}, we examined the
properties of second order extended quadrature method of moments
(EQMOM) proposed in \cite{vikas2013radiation} in slab geometry, and
the model was referred as the $B_2$ model. In EQMOM, the ansatz
$\hat I$ is reconstructed by a combination of beta distributions. The
beta distribution on $[-1, 1]$ is given by
$$
\mathcal{F}(\mu; \gamma, \delta) = \frac1{2 \Beta(\xi, \eta)}
 \left(\frac{1 + \mu}2 \right)^{\xi - 1}
 \left(\frac{1 - \mu}2 \right)^{\eta-1},
\quad 
 \xi  = \frac\gamma\delta, \quad
 \eta = \frac{1 - \gamma}\delta.
$$
where $\Beta(\xi, \eta)$ is the beta function. For the $B_2$ model in
1D slab geometry, the ansatz is taken as
\[
w \mathcal{F}(\mu; \gamma, \delta)
\]
where the parameters $w$, $\gamma$, and $\delta$ are given by
consistency to the known moments.

We found that the 1D $B_2$ model shares the key features of the $M_2$
model in slab geometry, including existence of non-negative ansatz and
therefore realizability, as well as global hyperbolicity. It is the
focus of this paper to extend the 1D $B_2$ model to three-dimensional
case.

Our motivation to this extension is based on observing a common
attribute between the $B_2$ and the $M_2$ ansatz in 1D slab
geometry. Both ansatz can exactly recover the isotropic distribution,
and at the same time, give a combination of Dirac functions on the
boundary of the realizability domain. Dirac functions could not be
recovered by the standard spectral method which has a polynomial as an
ansatz. It has been pointed out that the inability to capture
anisotropy is a drawback of the standard spectral method
\cite{frank2006partial}.

In three-dimensional space, the anisotropy of the specific intensity
could come in orthogonal directions. For example, we consider a setup
similar to the double beam problem discussed in
\cite{pichard2016approximation} \footnote{ In
  \cite{pichard2016approximation}, this example is used to demonstrate
  the advantage of the $M_2$ model over its first-order counterpart,
  the $M_1$ model.  }.  For the region $[x,y]\in[-1,1]\times[-1,1]$,
consider equation \eqref{eq:rt} with the right-hand side chosen as
isotropic scattering (which means $\sigs$ is a non-negative constant):
\[
  \cC(I) = \sigs \left( -I + \dfrac{1}{4\pi} \Vint{I} 
  \right).
\]
Laser beams are imposed as boundary inflow from orthogonal 
directions: $I = \delta(\bsOmega\cdot\be_x-1)$ on the 
boundary $x = -1$, and $I = \delta(\bsOmega\cdot\be_y - 1)$
on the boundary $y = -1$.

For the extreme case when the medium
is vacuum and $\sigs = 0$, the exact solution for any $ct > 2$
is
\begin{equation}\label{eq:crossing-beam-distribution}
  I = \delta(\bsOmega\cdot\be_x - 1) + 
  \delta(\bsOmega\cdot\be_y - 1).
\end{equation}
It is pointed out in \cite{pichard2016approximation} that 
the $M_2$ ansatz is able to exactly reproduce the distribution in
\eqref{eq:crossing-beam-distribution} from the moments. We aim to construct an
ansatz that can capture anisotropy in  
orthogonal directions, like the $M_2$ ansatz.

For non-vanishing scattering, the steady-state solution of the above 
problem is an isotropic distribution. For any period before steady-state
is reached, the exact specific intensity $I$ should be somewhere
between double beams, as in \eqref{eq:crossing-beam-distribution},
and isotropic. The ansatz of the $M_2$ model provides
a smooth interpolation between these two extremes, giving it an advantage 
in simulating such problems. We aim to propose an
ansatz with similar features. This will be discussed in the following sections.


%% file: model.tex
\section{3D $B_2$ Model}
\label{sec:approx}
For second order models, which are the subject of this paper, 
the set of realizable moments as given in \cite{kershaw1976flux} is
\begin{equation}\label{eq:realizablity-region}
  \begin{split}
    \cM =  \Big\{ &  \left(E^0, \bE^1, \bbrE^2\right) 
    \in \bbR \times \bbR^3 \times \bbR^{3\times3}, 
    ~~\text{s.t.}~~0 < E^0 = \Trace(\bbrE^2),\\  
    &  \|\bE^1\| \leq E^0, \,
    ~~\text{and}~~E^0 \bbrE^2 - \bE^1 \otimes \bE^1 
  ~~\text{symmetric non-negative} \Big\}.
\end{split}
\end{equation}
It is also referred to as the realizability domain. 
Our goal is to reconstruct an ansatz of the specific intensity
given moments within $\cM$.

We take the summation of three axisymmetric distributions as the
ansatz for the specific intensity:
\begin{equation}\label{eq:B2-ansatz}
  \bansatz(\bOmega) = \sum_{i = 1}^3 \dfrac1{2\pi} w_i 
  f(\bOmega\cdot\bR_i; \gamma_i, \delta_i),
\end{equation}
where $\bR_i$ are three mutually orthogonal unit vectors. We assume that the matrix
$\bbrR = \left[\bR_1, \bR_2, \bR_3\right]$ satisfy $\det(\bbrR) = 1$.
It is also
assumed that $f(\mu; \gamma, \delta)$ is a non-negative function of
$\mu$ with two shape parameters $\gamma$ and $\delta$, and
$\displaystyle \int_{-1}^1 f(\mu; \gamma, \delta) \dd\mu = 1$. All the
parameters in the ansatz, including $\bR_i$, $w_i$, $\gamma_i$, and
$\delta_i$, $i=1,2,3$, are functions of known moments and are
independent of $\bOmega$. We first discuss the properties of
\eqref{eq:B2-ansatz} for any arbitrary non-negative function
$f(\mu;\gamma,\delta)$ whose integral over $\mu\in[-1,1]$ is one.
To simplify the computing process in discussing the consistency conditions, we
first make the following observation which will be used later.
\begin{lemma}\label{lem:integral-orthogonal}
  For any permutation $l, m, k$ of $1, 2, 3$, $\forall
  n_l,n_m, n_k\in \mathbb{N}$, we have
  \[
    \begin{aligned}
      & \Vint{(\bOmega\cdot\bR_l)^{n_l} (\bOmega\cdot\bR_m)^{n_m} 
      (\bOmega\cdot\bR_k)^{n_k} f(\bOmega\cdot\bR_k; \gamma_k,
    \delta_k)} = \\ & \qquad \qquad \qquad \left
      \{
        \begin{array}{l}
          0, ~~\text{if either }~n_l \text{ or } n_m \text{ is odd}, \\
          [4mm]

          2 \pi \int_{-1}^1 \mu^{n_k} f(\mu;\gamma_k, \delta_k) \dd\mu, 
          ~~\text{if } n_l = n_m = 0, \\
          [4mm]

          \pi \int_{-1}^1 (1 - \mu^2) \mu^{n_k} f(\mu; \gamma_k,
          \delta_k) \dd\mu, ~~\text{if } n_l = 2, ~n_m = 0.
        \end{array}\right.
      \end{aligned}
    \]
  \end{lemma}
\begin{proof}
  Note that $f(\bOmega\cdot\bR_k; \gamma_k, \delta_k)$ is an 
  axisymmetric function with $\bR_k$ as its symmetric axis.
  Once $\bR_i$, $i=1,2,3$, are given, the value of
  \[
    \int_{\bbS^2} (\bOmega\cdot\bR_l)^{n_l} (\bOmega\cdot\bR_m)^{n_m}
    (\bsOmega \cdot \bR_k)^{n_k}
    f(\bOmega\cdot\bR_k; \gamma_k, \delta_k)
    \dd\bOmega,
  \]
  can be calculated conveniently by setting $\bR_k$ as coordinate
  axes. Below, we will repeatedly use this method to compute the moments. 

  Set the $z$-axis to be aligned to $\bR_k$, the $x$-axis to 
  be aligned to $\bR_l$, and the $y$-axis aligned to $\bR_m$.
  Then
\[
  \begin{split}
    & \int_{\bbS^2} (\bOmega\cdot\bR_l)^{n_l} (\bOmega\cdot\bR_m)^{n_m}
    (\bOmega\cdot\bR_k)^{n_k}
    f(\bOmega\cdot\bR_k; \gamma_k, \delta_k)
    \dd\bOmega \\
    = &  \int_0^\pi \left(\int_0^{2\pi} 
    (\sin\theta \cos\phi)^{n_l} (\sin\theta \sin\phi)^{n_m}
    (\cos\theta)^{n_k} f(\cos\theta; 
  \gamma_k,\delta_k) \dd\phi\right) \sin\theta \dd\theta \\
  = &  \int_0^\pi (\sin\theta )^{n_l + n_m} (\cos\theta)^{n_k}
  f(\cos\theta; \gamma_k,\delta_k) \sin\theta \dd\theta
  \int_0^{2\pi} (\cos\phi)^{n_l} (\sin\phi)^{n_m} \dd\phi. 
\end{split}
  \]
  Assume that $n_l$ is odd, let $n_l = 2 j + 1$.
  Then
  \[
    \begin{split}
      & \int_0^{2\pi} (\cos\phi)^{n_l} (\sin\phi)^{n_m} \dd\phi \\
      = & \int_0^{2\pi} (\cos\phi)^{2 j + 1} (\sin\phi)^{n_m} \dd\phi \\
      = & \int_0^{2\pi} (1-(\sin\phi)^2)^j (\sin\phi)^{n_m} \dd\sin\phi \\
      = & ~0.
    \end{split}
  \]
  Other cases are proved similarily.

  To calculate
  \[
    \dfrac1{2\pi} \int_{\bbS^2} (\bOmega\cdot\bR_k)^{n_k}
    f(\bOmega\cdot\bR_k; \gamma_k, \delta_k) \dd\bOmega,
  \]
  we set up the coordinate system such that $\bR_k$ is aligned to the
  $z$-axis. Let $\mu = \bOmega\cdot\bR_k$, then
  \[
    \begin{split}
      & \dfrac1{2\pi} \int_{\bbS^2} (\bOmega\cdot\bR_k)^{n_k}  
      f(\bOmega\cdot\bR_k; \gamma_k, \delta_k) \dd\bOmega \\
      = & \dfrac1{2\pi} \int_0^\pi \left( \int_0^{2\pi} 
    (\cos\theta)^{n_k} f(\cos\theta; \gamma_k, \delta_k) \dd\phi \right)
    \sin\theta \dd\theta \\
    = & \int_0^{\pi} (\cos\theta)^{n_k} f(\cos\theta; \gamma_k, \delta_k)   
    \sin\theta \dd\theta \\
    = & \int_{-1}^1 \mu^{n_k} f(\mu; \gamma_k, \delta_k) \dd\mu. 
  \end{split}
\]
If we let the $z$-axis to be aligned to $\bR_k$ and the $y$-axis
to be aligned to $\bR_l$, we have
\[
  \begin{split}
    & \dfrac1{2\pi} \int_{\bbS^2} (\bOmega\cdot\bR_l)^2
    (\bsOmega \cdot \bR_k)^{n_k}
    f(\bOmega\cdot\bR_k; \gamma_k, \delta_k) \dd\bOmega \\
    = & \dfrac1{2\pi} \int_0^\pi \left( \int_0^{2\pi}
    (\sin\theta\sin\phi)^2 (\cos\theta)^{n_k} f(\cos\theta; \gamma_k, \delta_k)
  \dd\phi \right) \sin\theta \dd\theta, \\
  = & \dfrac1{2\pi} \int_0^\pi (\sin\theta)^2 (\cos\theta)^{n_k} 
  f(\cos\theta; \gamma_k, \delta_k) \sin\theta \dd\theta
  \int_0^{2\pi} (\sin\phi)^2 \dd\phi \\
  = & \dfrac12 \int_{-1}^1 (1-\mu^2) \mu^{n_k} f(\mu; \gamma_k,
  \delta_k) \dd\mu. 
\end{split}
  \]
  Summarizing the results from the above three cases completes the proof of 
  this lemma.
\end{proof}
Take $\bv$ as defined in \eqref{eq:v-def}, the moments of interest are
\[
    \bE = [ E^0,E^1_1,E^1_2,E^1_3, E^2_{11},E^2_{12},E^2_{13},
    E^2_{22},E^2_{23} ]^T = \int_{\bbS^2} \bv \bansatz \dd\bOmega.
\]
The moment system based on the ansatz
\eqref{eq:B2-ansatz} is derived as
\begin{equation}\label{eq:moment_model}
  \pd{\bE}{t} + \pd{\bff_x(\bE)}{x} + \pd{\bff_y(\bE)}{y} +
  \pd{\bff_z(\bE)}{z} = \br(\bE),
\end{equation}
where
\[
  \begin{split}
    & \bff_x = [E^1_1, E^2_{11}, E^2_{12}, E^2_{13}, E^3_{111},
    E^3_{112}, E^3_{113}, E^3_{122}, E^3_{123}]^T = \int_{\bbS^2}
    (\bOmega\cdot\be_x) \bv \bansatz \dd\bOmega,\\
    & \bff_y = [E^1_2, E^2_{21}, E^2_{22}, E^2_{23}, E^3_{211},
    E^3_{212}, E^3_{213}, E^3_{222}, E^3_{223}]^T = \int_{\bbS^2}
    (\bOmega\cdot\be_y) \bv \bansatz \dd\bOmega,\\
    & \bff_z = [E^1_3, E^2_{31}, E^2_{32}, E^2_{33}, E^3_{311},
    E^3_{312}, E^3_{313}, E^3_{322}, E^3_{323}]^T = \int_{\bbS^2}
    (\bOmega\cdot\be_z) \bv \bansatz \dd\bOmega,
  \end{split}
\]
and $\br(\bE)$ is calculated from the scattering term, which is out of
the scope of our interests in this paper.

The parameters $\bR_i$, $w_i$, $\gamma_i$, and $\delta_i$ have to satisfy the 
consistency conditions: 
\begin{equation}\label{eq:consistency}
  E^0 = \Vint{\bansatz}, \quad
  \bE^1 = \Vint{\bOmega \bansatz}, \quand
  \bbrE^2 = \Vint{\bOmega \otimes \bOmega \bansatz}.
\end{equation}

The vectors $\bR_i$ in \eqref{eq:B2-ansatz} are determined by the consistency conditions
\eqref{eq:consistency} instantly, as shown in the following lemma. 
\begin{lemma}\label{lem:Rj}
  The consistency constraints \eqref{eq:consistency} require that
  $\bR_j$, $j=1,2,3$, be the eigenvectors of $\bbrE^2$.
\end{lemma}
\begin{proof}
  Let 
  $\bbrR = \left[\bR_1, \bR_2, \bR_3\right]$. As $\bbrR$ is an orthogonal 
  matrix, we have $\bbrR^{-1} = \bbrR^T$. To prove the lemma, it suffices to show that
  \begin{equation}\label{eq:consistency-Rj}
    \bR_j^T\bbrE^2\bR_i = \Vint{(\bOmega\cdot\bR_j)
    (\bOmega\cdot\bR_i)\bansatz} = 0,\quad \text{if}~j\not=i.
  \end{equation}
  The reason is that \eqref{eq:consistency-Rj} would indicate that
  $\bbrR^{-1}\bbrE^2\bbrR$ is a diagonal matrix, and therefore $\bR_j$,
  $j=1,2,3$, are the eigenvectors of $\bbrE^2$.
  
  In order to prove \eqref{eq:consistency-Rj}, consider the case $j=1$, $i=2$,
  \[
    \bR_1^T\bbrE^2\bR_2  = \Vint{(\bOmega\cdot\bR_1)
    (\bOmega\cdot\bR_2) \bansatz}
    = \dfrac1{2\pi} \int_{\bbS^2} (\bOmega\cdot\bR_1) (\bOmega\cdot\bR_2) 
    \sum_{i = 1}^3 w_i f(\bOmega\cdot\bR_i; 
    \gamma_i, \delta_i) \dd\bOmega.
  \]
  By Lemma \ref{lem:integral-orthogonal},
  \[
    \int_{\bbS^2} (\bOmega\cdot\bR_1) (\bOmega\cdot\bR_2) 
    w_k f(\bOmega\cdot\bR_k; \gamma_k, \delta_k)
    \dd\bOmega = 0, \quad k = 1,2,3.
  \]
  Therefore $\bR_1^T\bbrE^2\bR_2 = 0$. Similar arguments
  show that $\bR_1^T\bbrE^2\bR_3 = \bR_2^T\bbrE^2\bR_3 = 0$. 
\end{proof}

With the parameters $\bR_i$ determined, we now consider the consistency requirements under the coordinate
system $(\bR_1,\bR_2,\bR_3)$. In this coordinate system, $\bbrE^2$ is
a diagonal matrix. Also, as Lemma \ref{lem:Rj} specify $\bR_j$, 
$j=1,2,3$, to be the eigenvectors of $\bbrE^2$, consistency of all 
the non-diagonal elements of $\bbrE^2$ are naturally satisfied. 
Therefore we only need to look at the consistency of $E^0$, $\bE^1$, and all
the eigenvalues of $\bbrE^2$. This leaves us with 6 constraints. On the 
other hand, with $\bR_j$, $j = 1,2,3$ fixed, there are 9 parameters in the 
ansatz \eqref{eq:B2-ansatz}.
Denote
\begin{equation}\label{eq:sigma-def}
  \sigma_i = w_i \int_{-1}^1 \mu^2 f(\mu;\gamma_i,\delta_i)
  \dd\mu.
\end{equation}
The following lemma shows that once $\sigma_i$, $i = 1,2,3$ are specified, then $w_i$ 
for $i = 1,2,3$ would be determined by consistency constraints.
\begin{lemma}
  Let $\lambda_i$ be the eigenvalue corresponding to $\bR_i$. Then
  $w_i$, $\sigma_i$ and $\lambda_i$ satisfy the following constraints:
  \begin{equation}\label{eq:w-cond}
  \begin{split}
    & w_1 = 2\sigma_1 - (\sigma_2 + \sigma_3) -\lambda_1
    +\lambda_2 + \lambda_3,\\
    & w_2 = 2\sigma_2 - (\sigma_1 + \sigma_3) - \lambda_2
    +\lambda_1 + \lambda_3,\\
    & w_3 = 2\sigma_3 - (\sigma_1 + \sigma_2) - \lambda_3
    +\lambda_1 + \lambda_2.
  \end{split}
\end{equation}
\end{lemma}
\begin{proof}
  Firstly, 
  \[
  \lambda_1 = \Vint{(\bOmega\cdot\bR_1)^2 \bansatz} = \dfrac1{2\pi}
  \int_{\bbS^2} (\bOmega\cdot\bR_1)^2 \sum\limits_{i = 1}^3 w_i 
  f(\bOmega\cdot\bR_i; \gamma_i, \delta_i) \dd\bOmega.
  \]
  By Lemma \ref{lem:integral-orthogonal},
  \[
    \dfrac1{2\pi} \int_{\bbS^2} (\bOmega\cdot\bR_1)^2 w_1 
    f(\bOmega\cdot\bR_1; \gamma_1, \delta_1) \dd\bOmega 
    = w_1 \int_{-1}^1 \mu^2 f(\mu; \gamma_1, 
    \delta_1) \dd\mu 
    = \sigma_1.
  \]
  For $k = 2,3$, again by Lemma \ref{lem:integral-orthogonal},
  \[
    \begin{split}
      & \dfrac1{2\pi} \int_{\bbS^2} (\bOmega\cdot\bR_1)^2 w_k
      f(\bOmega\cdot\bR_k; \gamma_k, \delta_k) \dd\bOmega  \\
      = & \dfrac12 w_k \int_{-1}^1 (1-\mu^2) f(\mu; \gamma_k,
      \delta_k) \dd\mu 
      = \dfrac12 (w_k - \sigma_k).
    \end{split}
  \]
  Therefore, 
  \begin{equation}\label{eq:lambda1-consistency}
    \lambda_1 = \sigma_1 + \dfrac12 (w_2 - \sigma_2) + 
    \dfrac12 (w_3 - \sigma_3).
  \end{equation}
  By symmetry, we have
  \begin{equation}\label{eq:lambda23-consistency}
    \begin{split}
      & \lambda_2 = \sigma_2 + \dfrac12 (w_1 - \sigma_1) + 
      \dfrac12 (w_3 - \sigma_3), \\
      & \lambda_3 = \sigma_3 + \dfrac12 (w_1 - \sigma_1) + 
      \dfrac12 (w_2 - \sigma_2).
    \end{split}
  \end{equation}
  Rewriting \eqref{eq:lambda1-consistency} and
  \eqref{eq:lambda23-consistency} by solving the equations as a linear
  system of $w_i$ yields the final results \eqref{eq:w-cond}.
\end{proof}
Once $w_i$, $i=1,2,3$, are given, consistency
requires that $\gamma_i$ and $\delta_i$ satisfy
\begin{equation}\label{eq:moment-problem-1D}
  w_i \int_{-1}^1 \mu f(\mu;\gamma_i,\delta_i)
  \dd\mu = F_i,
\end{equation}
where $F_i = \bE^1\cdot\bR_i$. If $w_i = 0$, then the term $w_i 
f(\bOmega \cdot \bR_i; \gamma_i, \delta_i)$
does not appear in the ansatz \eqref{eq:B2-ansatz}. From now on we
assume $w_i \not= 0$. Recall that by definition, the function $f(\mu; \gamma, \delta)$ 
is a non-negative distribution on $\mu \in [-1,1]$, and its zeroth moment is $1$. Moreover,
the first and second-order moments of $f$ are respectively 
$\dfrac{F_i}{w_i}$ and $\dfrac{\sigma_i}{w_i}$. So,
combining \eqref{eq:sigma-def} and \eqref{eq:moment-problem-1D} define
a 1D moment problem. This means that once the value of the three parameters $\sigma_i$, 
$i = 1,2,3$ are specified, the consistency condition \eqref{eq:consistency} could 
be decomposed into three decoupled 1D moment problems
\begin{equation}\label{eq:consistency-decoupled}
  \left\{
  \begin{array}{l}
     \int_{-1}^1 \mu f(\mu;\gamma_i,\delta_i)
     \dd\mu = \dfrac{F_i}{w_i}, \\ [5mm]
     \int_{-1}^1 \mu^2 f(\mu;\gamma_i,\delta_i)
    \dd\mu = \dfrac{\sigma_i}{w_i},
  \end{array}\right.
\end{equation}
for $i = 1,2,3$.
According to \cite{CurFial91}, the realizability domains
of the 1D moment problems in \eqref{eq:consistency-decoupled}
are: 
\begin{equation}\label{eq:condition-realizable}
  \left(\dfrac{F_i}{w_i}\right)^2
  \leq\dfrac{\sigma_i}{w_i}\leq 1, \quad i = 1,2,3.
\end{equation}

A sufficient condition for the existence of non-negative
ansatz $\bansatz$ is $w_i \geq 0$, $i = 1,2,3$. It follows that 
a non-negative ansatz $\bansatz$ exists under the following conditions:
\begin{equation}\label{eq:condition-positive-ansatz}
  \left(\dfrac{F_i}{w_i}\right)^2
  \leq\dfrac{\sigma_i}{w_i}\leq 1, \quad
  w_i \geq 0, \quad i = 1,2,3.
\end{equation}
One would like to give a non-negative ansatz for as large a part of
the realizability domain as possible to have a realizable closure.
Before examining the non-negativity of the ansatz
\eqref{eq:B2-ansatz}, we give the following result, which is an
alternative characterization of the realizable moments:

\begin{lemma}\label{proposition-realizability}
  Let $\{\lambda_j,\bR_j\}$, $j=1,2,3$, be the eigenpairs of
  $\bbrE^2$, and $F_j = \bE^1\cdot\bR_j$. Then the realizability domain $\cM$
  given by \eqref{eq:realizablity-region} is
  \begin{equation}\label{eq:realizability-region-alternative}
    \cM = \left\{ \left(E^0, \bE^1, \bbrE^2\right) \, \left| \,
      0 < \sum\limits_{i=1}^3 \lambda_i = E^0,\, 
      \sum\limits_{i=1}^3 \dfrac{F_i^2}{\lambda_i} \leq
    E^0 \right.\right\}.
  \end{equation}
  In \eqref{eq:realizability-region-alternative},
the term $\dfrac{F_i^2}{\lambda_i} = 0$ is taken to be zero if $\lambda_i = 0$.
\end{lemma}
\begin{proof}
Denote the normalized first and second-order moments by
$\hat{\bE}^1 = \dfrac{\bE^1}{E^0}$ and $\hat{\bbrE}^2 = \dfrac{\bbrE^2}{E^0}$.
Let
\[
  \Lambda = \diag\left\{\dfrac{\lambda_1}{E^0}, \dfrac{\lambda_2}{E^0}, 
  \dfrac{\lambda_3}{E^0}\right\}, \quad
  \Lambda^{\frac12} = \diag\left\{\sqrt{\dfrac{\lambda_1}{E^0}}, \sqrt{\dfrac{\lambda_2}{E^0}}, 
  \sqrt{\dfrac{\lambda_3}{E^0}} \right\},
\]
and denote
\[
  \bbrR = [\bR_1, \bR_2, \bR_3], \quad \bbrT = \Lambda^{\frac12} \bbrR.
\]
Then
\[
  \hat{\bbrE}^2 = \bbrR^T \Lambda \bbrR = \bbrT^T \bbrT. 
\]
Assuming that $\lambda_i \not= 0$, $i = 1,2,3$. Then non-negativity of the matrix 
$\hat{\bbrE}^2 - \hat{\bE}^1 \otimes \hat{\bE}^1$
is equivalent to the non-negativity of the matrix $
  \Identity - \bbrT^{-T} \hat{\bE}^1 (\hat{\bE}^1)^T \bbrT^{-1}$,
which, in turn, is equivalent to 
$\|(\hat{\bE}^1)^T \bbrT^{-1}\|_2 \leq 1$, and therefore equivalent to
\begin{equation}\label{eq:realizable-cond}
  \sum\limits_{i = 1}^3 \dfrac{F_i^2}{\lambda_i} \leq E^0.
\end{equation}
The cases when there exists $i$ for which $\lambda_i = 0$ can be proved by entirely similar arguments.
\end{proof}
\begin{remark}
The above lemma could also be proved by applying the method for solving modified
eigenvalue problems proposed in \cite{yu1991recursive}. 
\end{remark}
Making use of Lemma \ref{proposition-realizability},
the realizability domain can be visualized as: take any point inside 
a triangle and let $\left(\dfrac{\lambda_1}{E^0}, \dfrac{\lambda_2}{E^0},
\dfrac{\lambda_3}{E^0}\right)$ be its barycentric coordinates.  
Then the corresponding $(F_1,F_2,F_3)$ lie in the ellipsoid 
\eqref{eq:realizable-cond}. 
Each side of the triangle corresponds to the cases where at least 
one eigenvalue of $\bbrE^2$ vanishes. In such cases non-negativity of
$\bansatz$ given in \eqref{eq:B2-ansatz} would impose the following
constraints on the first and second-order moments:
\begin{lemma}\label{lem:lambda2sigma}
  The non-negativity of the ansatz $\bansatz$ requires
  that if there exists $i = 1,2$ or $3$ such that $\lambda_i = 0$, then
  \[
  \left\{
    \begin{array}{l}
      F_i=0 \text{ and } \sigma_i = 0, \\
      |F_j| \leq \sigma_j \leq \lambda_j,
      \quad \text{for}~~\forall j \not= i.
    \end{array}\right.
  \]
  \end{lemma}
\begin{proof}
  Consider the case $i=1$. Since
  \begin{equation}\label{eq:cond-lambda-zero}
    0 = \lambda_1 = \int_{\bbS^2} (\bOmega \cdot \bR_1)^2
    \bansatz(\bOmega) \dd \bOmega,
  \end{equation}
  then $\bansatz$ can only be non-zero when $\bOmega \cdot \bR_1 = 0$.
  This gives
  \[
    F_1 = \int_{\bbS^2} (\bOmega \cdot \bR_1) \bansatz(\bOmega)
    \dd \bOmega = 0.
  \]
  To prove $\sigma_1 = 0$, let us study two cases.  
  \begin{enumerate}
    \item For $w_1 = 0$, it can be seen from \eqref{eq:sigma-def} 
      that $\sigma_1 = 0$.
    \item If $w_1 \not= 0$. Recall that $\bansatz$ can only be non-zero when 
      $\bOmega \cdot \bR_1 = 0$. Then \eqref{eq:sigma-def} shows $\sigma_1 = 0$. 
  \end{enumerate}
  
  Next, we show that $\sigma_j \geq |F_j|$, $j = 2,3$. We look at 
  two cases.
  \begin{enumerate}
    \item In the case that $w_j = 0$, by \eqref{eq:moment-problem-1D} we
      have $F_j = 0$, and by \eqref{eq:sigma-def} we see that $\sigma_j = 0$.
      Hence $\sigma_j \geq |F_j|$.
    \item If $w_j \not= 0$. Again, note that $\bansatz$ can only be non-zero 
      when $\bOmega \cdot \bR_1 = 0$, therefore for $j \not= 1$,
      the function $f(\bOmega\cdot\bR_j; \gamma_j, \delta_j)$
      has the form 
      $$f(\bOmega\cdot\bR_j; \gamma_j, \delta_j) = \alpha_j^-
      \delta(\bOmega \cdot \bR_j + 1) + \alpha_j^+ \delta(\bOmega \cdot \bR_j - 1).$$
      From \eqref{eq:sigma-def} we know that in such cases $\sigma_j = w_j$. 
      Combine this with the left inequality in 
      \eqref{eq:condition-positive-ansatz}, and we have $\sigma_j \geq |F_j|$.
  \end{enumerate}

  Finally, we prove $\sigma_j \leq \lambda_j$, $j = 2,3$. 
  Plugging $\sigma_1 = \lambda_1 = 0$ into \eqref{eq:w-cond} gives
  \[
    \sigma_2 - \sigma_3 = \lambda_2 - \lambda_3,
  \]
  and
  \[
    \sigma_2 + \sigma_3 = \lambda_2 + \lambda_3 - w_1,
  \]
  If $\bansatz$ is non-negative, then $w_1 \geq 0$. Combine the above and notice
  that $\lambda_2 + \lambda_3 = E^0$, we have
  \[
    \sigma_j \leq \lambda_j, \quad j = 2,3.
  \]
  The proofs for $i=2,3$, follows in a similar manner.
\end{proof}
\begin{remark}
As a special case of Lemma \ref{lem:lambda2sigma}, if there exists $j$ such that 
$\lambda_j = E^0$, and $\lambda_i = 0$ for $i \not= j$, then 
  \[
    \left\{
      \begin{array}{l}
        |F_j| \leq \sigma_j \leq \lambda_j = E^0, \\
        F_i=0 \text{ and } \sigma_i = 0,\quad\text{for}~~\forall i \not= j. 
      \end{array}\right.
    \]
  \end{remark}
From Lemma \ref{lem:lambda2sigma}, it is clear that when $\lambda_i = 0$ is the only zero
eigenvalue of $\bbrE^2$, the region for which the ansatz \eqref{eq:B2-ansatz} 
admits a non-negative distribution is limited to the rectangle $|F_j| \leq \lambda_j$,
$j\not=i$. We point out that this rectangle can cover only 4 points for the boundary of
the realizability domain in \eqref{eq:realizable-cond}, which in this case 
becomes the ellipse
\[
  \dfrac{F_j^2}{\lambda_j} + \dfrac{F_k^2}{\lambda_k} = E^0,
  \quad j \not= k.
\]

For other boundary moments, we have the following result:
\begin{lemma}\label{lem:realizable-bound}
  Suppose $\lambda_i > 0$, $i = 1,2,3$.
  Then on the boundary of the realizability domain, where
  \begin{equation}\label{eq:cond-boundary}
    \dfrac{F_1^2}{\lambda_1} + \dfrac{F_2^2}{\lambda_2}
    + \dfrac{F_3^2}{\lambda_3} = E^0,
  \end{equation}
  there are only two kinds of moments for which $\bansatz$ can be non-negative:
  \begin{enumerate}
    \item   $\exists i$, such that $\lambda_i = \dfrac{F_i^2}{E^0}$.
      Meanwhile for $j,k\not=i$, the relationships $\lambda_j = \lambda_k$  and
      $F_j = F_k = 0$ hold.
    \item $\forall j = 1,2,3$, the constraint $|F_j| = \lambda_j$ is satisfied. 
  \end{enumerate}
\end{lemma}
\begin{proof}
  Let the covariance matrix of the distribution function 
  be
  \[
    \bbrV = \dfrac{\bbrE^2}{E^0}-\left(\dfrac{\bE^1}{E^0}
    \right)\left(\dfrac{\bE^1}{E^0}\right)^T.
  \]
  If
  \[
    \dfrac{F_1^2}{\lambda_1} + \dfrac{F_2^2}{\lambda_2}
    + \dfrac{F_3^2}{\lambda_3} = E^0,
  \]
  then there exists at least one zero eigenvalue for $\bbrV$.
  Denote the corresponding eigenvector by $\bU$, and 
  \cite{kershaw1976flux} has shown that
  any non-negative distribution could be non-zero only when 
  $\bOmega \cdot \bU = \dfrac1{E^0}(\bE^1 \cdot \bU)$. We
  will repeatedly make use of this fact in the following 
  discussions.

  We study the two possible cases:
  \begin{enumerate}
    \item Suppose $\bU$ is aligned with some eigenvector of $\bbrE^2$. 
      Without loss of generality, we assume $\bR_3\slash\slash\bU$. 
      Then a non-negative distribution could be non-zero only on 
      $\bOmega\cdot\bR_3 = \dfrac{F_3}{E^0}$. In addition,
      \begin{equation}\label{eq:eigenvector-V}
        0 = \bU^T \bbrV \bU = \bR_3^T \left[\dfrac1{E^0} \bbrE^2 - 
          \left(\dfrac{\bE^1}{E^0} \right) \left(\dfrac{\bE^1}{E^0}
        \right)^T\right] \bR_3 = \dfrac{\lambda_3}{E^0} - \left(
        \dfrac{F_3}{E^0} \right)^2,
      \end{equation}
      which gives $\lambda_3 = \dfrac{F_3^2}{E^0}$.
      If $F_3 = 0$ then $\lambda_3 = 0$, which has been 
      ruled out in our assumptions. So $F_3 \not= 0$, which means a non-negative 
      distribution \eqref{eq:B2-ansatz} can only be
      \begin{equation}\label{eq:single-delta-dist}
        \bansatz = \dfrac{E^0}{2 \pi}
        \delta\left(\bOmega\cdot\bR_3 - \dfrac{F_3}{E^0}\right).
      \end{equation}
      Therefore $w_1 = w_2 = 0$, and by \eqref{eq:sigma-def} and 
      \eqref{eq:moment-problem-1D} we would have
      $\sigma_1 = \sigma_2 = F_1 = F_2 = 0$. Substituting this into \eqref{eq:w-cond}
      gives $\lambda_1 = \lambda_2 = \dfrac12(w_3 - \lambda_3)$.
      Conversely, moments satisfying $\lambda_1 = \lambda_2$ and $F_1 = F_2 = 0$
      in addition to $\lambda_3 = \dfrac{F_3^2}{E^0}$ could be generated by the ansatz 
      \eqref{eq:single-delta-dist}. 

    \item Consider the case when $\bU$ is not aligned to any $\bR_j$. 
      The only way to give a non-negative 
      distribution for \eqref{eq:B2-ansatz} in this case
      is
      \[
        \bansatz = \sum\limits_{i = 1}^3 [ \alpha_i^+ \delta(\bOmega \cdot \bR_i
        - 1) + \alpha_i^- \delta(\bOmega \cdot \bR_i + 1) ].
      \]
      Hence $\sigma_j = w_j $, $j=1,2,3$. Combining these with \eqref{eq:w-cond}
      gives $\sigma_j = \lambda_j$, $j = 1,2,3$.
      But condition \eqref{eq:condition-positive-ansatz} 
      require
      \begin{equation}
        \left|F_j\right| \leq \lambda_j,\quad
        j = 1,2,3.
      \end{equation}
      Recall assumption \eqref{eq:cond-boundary}, and
      notice  
      \begin{equation}
        E^0 = \dfrac{F_1^2}{\lambda_1}
        +\dfrac{F_2^2}{\lambda_2}
        +\dfrac{F_3^2}{\lambda_3}
        \leq\lambda_1+\lambda_2+\lambda_3=E^0.
      \end{equation}
      For all inequalities to hold, we need 
      \begin{equation}\label{cond-boundary-ellipsoid}
        \left|F_j\right|=\lambda_j,~~j=1,2,3.
      \end{equation}
      Conversely, for moments satisfying condition \eqref{cond-boundary-ellipsoid}, 
      choosing 
      \begin{equation}
        \sigma_j = \lambda_j,\quad j = 1,2,3.
      \end{equation}
      would give a non-negative ansatz.
  \end{enumerate}
  The proof is completed.
\end{proof}

We now turn to specifying the formula for $f$. We take $f$
to be the beta distribution used in the $B_2$ ansatz for slab geometry
\begin{equation}\label{eq:beta-dist}
  f(\mu; \gamma, \delta) = \mathcal{F}(\mu; \gamma, \delta),
  \quad 
  \xi  = \frac\gamma\delta, \quad
  \eta = \frac{1 - \gamma}\delta.
\end{equation}

Retaining only one term in \eqref{eq:B2-ansatz} would provide the same 
ansatz as the one-dimensional $B_2$ ansatz which we studied previously 
\cite{alldredge2016approximating}. Taking $\xi = \eta = 1$ in equation
\eqref{eq:beta-dist} would give $f$ as a constant function. If either $\xi$
or $\eta$ approach zero, the limit of the function $f$ is a Dirac function. 
If both of them go to zeros at a fixed rate, the function $f$ will become 
a combination of two Dirac functions. This capacity of \eqref{eq:beta-dist}
to interpolate between the constant function and Dirac functions is a 
feature it shares with the $M_2$ ansatz. Also, for slab geometry, the $B_2$ 
model possesses numerous nice properties similar to the $M_2$ model;
therefore, we use it as building blocks for three-dimensional ansatz.

If \eqref{eq:beta-dist} is the distribution function $f$ in 
\eqref{eq:sigma-def} and \eqref{eq:moment-problem-1D}, then for
$\sigma_i$, $w_i$ and $F_i$ satisfying the realizability condition 
\eqref{eq:condition-realizable}, we have  
\[
\xi_i \geq 0,\quad \eta_i \geq 0,
\]
which gives an integrable function for \eqref{eq:beta-dist}.
For the above cases, the parameters $\gamma_i$ and $\delta_i$ are given 
as follow:
\begin{lemma}\label{lem:gamma-delta}
  If \eqref{eq:condition-positive-ansatz} is fulfilled, we have
  \begin{equation}\label{eq:gamma-cond}
    \gamma_i = \frac{F_i  + w_i }{ 2 w_i} \quand 
    \delta_i = -\frac{ F_i^2 
      - \sigma_i w_i}{w_i^2 - \sigma_i w_i},\quad
    \forall i = 1,2,3.
  \end{equation}
\end{lemma}
\begin{proof}
  Note that the standard $\beta$ distribution
  $\dfrac1{\Beta(\xi,\eta)} x^{\xi - 1} (1 - x)^{\eta - 1}$ has the
  properties \cite{johnson1970discrete}:
  \[
  \begin{split}
    & \int_0^1 x \dfrac1{\Beta(\xi,\eta)} x^{\xi - 1}
    (1 - x)^{\eta - 1} = \dfrac{\xi}{\xi + \eta}, \\
    & \int_0^1 x^2 \dfrac1{\Beta(\xi,\eta)} x^{\xi - 1}
    (1 - x)^{\eta - 1} = \dfrac{\xi (\xi + 1)}
    {(\xi + \eta)(\xi + \eta + 1)}.
  \end{split}
  \]
  Therefore,
  \begin{equation}\label{eq:1D-1moment}
    \begin{split}
      & \int_{-1}^1 \mu f(\mu;\gamma_i,\delta_i)
      \dd\mu  
      \xlongequal{x = \frac12 (1 + \mu)} 
      \int_0^1 (2x - 1) \dfrac1{\Beta(\xi_i, \eta_i)} 
      x^{\xi_i - 1} (1-x)^{\eta_i - 1} \dd x\\
      = & 2\int_0^1 x \dfrac1{\Beta(\xi_i, \eta_i)}
      x^{\xi_i - 1} (1 - x)^{\eta_i - 1} \dd x
      - 1 
      = 2 \dfrac{\xi_i}{\xi_i + \eta_i} - 1 
      = 2 \gamma_i - 1.
    \end{split}
  \end{equation}
  Also,
  \begin{equation}\label{eq:1D-2moment}
    \begin{split}
      & \int_{-1}^1 \mu^2 f(\mu;\gamma_i,\delta_i)
      \dd\mu  
      \xlongequal{x = \frac12 (1 + \mu)} 
      \int_0^1 (2x - 1)^2 \dfrac1{\Beta(\xi_i, \eta_i)} 
      x^{\xi_i - 1} (1-x)^{\eta_i - 1} \dd x\\
      = & 4\int_0^1 x^2 \dfrac1{\Beta(\xi_i, \eta_i)}
      x^{\xi_i - 1} (1 - x)^{\eta_i - 1} \dd x
      - 4\int_0^1 x \dfrac1{\Beta(\xi_i, \eta_i)}
      x^{\xi_i - 1} (1 - x)^{\eta_i - 1} \dd x
      + 1 \\
      = & 4 \dfrac{\xi_i (\xi_i + 1)}{(\xi_i + \eta_i)
        (\xi_i + \eta_i + 1)} -4 \dfrac{\xi_i}{\xi_i + \eta_i}
      + 1 
      = 4 \dfrac{\gamma_i(\gamma_i - 1)}{1 + \delta_i} + 1.
    \end{split}
  \end{equation}
  Combining \eqref{eq:1D-1moment}, \eqref{eq:1D-2moment} with
  \eqref{eq:sigma-def}, \eqref{eq:moment-problem-1D}  
  gives us \eqref{eq:gamma-cond}.
\end{proof}
Note that \eqref{eq:sigma-def}, \eqref{eq:moment-problem-1D}, and
\eqref{eq:w-cond} together are the necessary and sufficient conditions 
for consistency constraints to all known
moments. This leaves $\sigma_i$, $i=1,2,3$, to be the three free
parameters. We shall return to the problem of determining $\sigma_i$
later. For the present, we assume $\sigma_i$, $i = 1,2,3$ are all given, and the
following lemma gives the closure relationship of the $B_2$ model.
\begin{lemma}\label{lem:E3}
  Let
  $\mathrm{\boldsymbol{R}} = [\bR_1, \bR_2, \bR_3] \in \bbR^{3\times
    3}$,
  and denote by $R_{ij}$ the entries of the matrix
  $\mathrm{\boldsymbol{R}}$, the flux closure is then given by
  $\bff(E^0, \bE^1, \bbrE^2)$, which relies on $\bE^1$, $\bbrE^2$ and
  $\bbrE^3$, with $\bbrE^3$ given as
\[
E^3_{ijk} =   \Vint{ (\bOmega\cdot\bR_l)
  (\bOmega\cdot\bR_m)
  (\bOmega\cdot\bR_n) \bansatz}
R_{il} R_{jm} R_{kn},
\]
where the Einstein summation convention is used.
  For distribution ansatz $\bansatz$ given by \eqref{eq:B2-ansatz},
  \begin{equation}\label{eq:B2-closure}
    \begin{aligned}
      &\Vint{ (\bOmega\cdot\bR_l) (\bOmega\cdot\bR_m)
    (\bOmega\cdot\bR_n) \bansatz} = \\ & \qquad \qquad \qquad \left
      \{
        \begin{array}{l}
          \dfrac{ F_l (\sigma_l^2 
          + 2 F_l^2 - 3 w_l \sigma_l)}
          {2 F_l^2 - w_l \sigma_l - w_l^2},
          ~~\text{if } l = m = n,\\ [4mm]
          \dfrac{ F_l }{2}\left(1-
            \dfrac{ \sigma_l^2 
            + 2 F_l^2 - 3 w_l \sigma_l }
            { 2 F_l^2 - w_l \sigma_l - w_l^2 }
          \right),~~\text{if }m=n,~m\not=l,\\ [4mm]
          0,~~\text{if}~l\not=m\not=n.
        \end{array}\right.
      \end{aligned}
    \end{equation}
  \end{lemma}
  \begin{proof}
    Consider the case when $l = m = n = 1$ at first.
    \[
      \Vint{ (\bOmega\cdot\bR_l) (\bOmega\cdot\bR_m)
      (\bOmega\cdot\bR_n) \bansatz} = 
      \dfrac1{2\pi} \int_{\bbS^2} (\bOmega\cdot\bR_1)^3  
      \sum_{i = 1}^3 w_i f(\bOmega\cdot\bR_i; 
      \gamma_i, \delta_i) \dd\bOmega.
    \]
    By Lemma \ref{lem:integral-orthogonal},
    \[
      \int_{\bbS^2} (\bOmega\cdot\bR_1)^3 w_2 f(\bOmega\cdot\bR_2;
      \gamma_2, \delta_2) \dd\bOmega = \int_{\bbS^2}
      (\bOmega\cdot\bR_1)^3 w_3 f(\bOmega\cdot\bR_3; \gamma_3, \delta_3)
      \dd\bOmega = 0,
    \]
    and
    \[
      \dfrac1{2\pi} \int_{\bbS^2} (\bOmega\cdot\bR_1)^3  
      w_1 f(\bOmega\cdot\bR_1; \gamma_1, \delta_1)
      \dd\bOmega = w_1 \int_{-1}^1
      \mu^3 f(\mu;\gamma_1,\delta_1) \dd\mu.
    \]
Note that the standard $\beta$ distribution
$\dfrac1{\Beta(\xi,\eta)} x^{\xi - 1} (1 - x)^{\eta - 1}$ has the
property \cite{johnson1970discrete}:
\[
  \int_0^1 x^3 \dfrac1{\Beta(\xi,\eta)} x^{\xi - 1}
  (1 - x)^{\eta - 1} = \dfrac{\xi (\xi + 1) (\xi + 2)}
  {(\xi + \eta)(\xi + \eta + 1)(\xi + \eta + 2)}.
\]
Therefore,
\[
  \begin{split}
    & \int_{-1}^1 \mu^3 f(\mu;\gamma_1,\delta_1)
    \dd\mu  
    \xlongequal{x = \frac12 (1 + \mu)} 
    \int_0^1 (2x - 1)^3 \dfrac1{\Beta(\xi_1, \eta_1)} 
    x^{\xi_1 - 1} (1-x)^{\eta_1 - 1} \dd x\\
    = & 8\int_0^1 x^3 \dfrac1{\Beta(\xi_1, \eta_1)}
    x^{\xi_1 - 1} (1 - x)^{\eta_1 - 1} \dd x
    - 12\int_0^1 x^2 \dfrac1{\Beta(\xi_1, \eta_1)}
    x^{\xi_1 - 1} (1 - x)^{\eta_1 - 1} \dd x\\
    & + 6\int_0^1 x \dfrac1{\Beta(\xi_1, \eta_1)}
    x^{\xi_1 - 1} (1 - x)^{\eta_1 - 1} \dd x
    - 1 \\
    = & \dfrac{8 \xi_1 (\xi_1 + 1) (\xi_1 + 2)}{(\xi_1 + \eta_1)
  (\xi_1 + \eta_1 + 1) (\xi_1 + \eta_1 + 2)} -  
  \dfrac{12 \xi_1 (\xi_1 + 1)}{(\xi_1 + \eta_1)
  (\xi_1 + \eta_1 + 1)} + \dfrac{6 \xi_1}{\xi_1 + \eta_1}
  - 1\\ 
  =  & \dfrac{(\xi_1 - \eta_1) (\xi_1^2 - 2 \xi_1 \eta_1 + 3 \xi_1
  + \eta_1^2 + 3 \eta_1 + 2)}{(\xi_1 + \eta_1) (\xi_1 + \eta_1 + 1)
  (\xi_1 + \eta_1 + 2)}.
\end{split}
  \]
  Recall \eqref{eq:beta-dist} and Lemma \ref{lem:gamma-delta}
  for the values of $\xi_1$ and $\eta_1$, we have
  \[
    \begin{split}
      & \int_{-1}^1 \mu^3 f(\mu;\gamma_1,\delta_1)
      \dd\mu \\
      = & \dfrac{(\xi_1 - \eta_1) (\xi_1^2 - 2 \xi_1 \eta_1 + 3 \xi_1
    + \eta_1^2 + 3 \eta_1 + 2)}{(\xi_1 + \eta_1) (\xi_1 + \eta_1 + 1)
    (\xi_1 + \eta_1 + 2)} \\
    = & \dfrac{ F_1 (\sigma_1^2 
  + 2 F_1^2 - 3 w_1 \sigma_1)}
  {2 F_1^2 - w_1 \sigma_1 - w_1^2}.
\end{split}
      \]
For $l = m = n = 2$ or $l = m = n = 3$ the computation is similar.

Now consider the case when $m = n$, $m\not=l$. Suppose $m = n = 1$ and $l = 2$.
Let us start by proving
\[
  \Vint{ (\bOmega\cdot\bR_1)^2 (\bOmega\cdot\bR_2) \bansatz }=  
  \Vint{ (\bOmega\cdot\bR_3)^2 (\bOmega\cdot\bR_2) \bansatz }.
\]
First, compute $\Vint{ (\bOmega\cdot\bR_1)^2 (\bOmega\cdot\bR_2) \bansatz }$.
\[
  \Vint{ (\bOmega\cdot\bR_1)^2 (\bOmega\cdot\bR_2)
  \bansatz} = 
  \dfrac1{2\pi} \int_{\bbS^2} (\bOmega\cdot\bR_1)^2
  (\bOmega\cdot\bR_2)
  \sum_{i = 1}^3 w_i f(\bOmega\cdot\bR_i; 
  \gamma_i, \delta_i) \dd\bOmega.
\]
Again by Lemma \ref{lem:integral-orthogonal},
\[
  \begin{split}
    & \int_{\bbS^2} (\bOmega\cdot\bR_1)^2
    (\bOmega\cdot\bR_2)
    w_3 f(\bOmega\cdot\bR_3; \gamma_3, \delta_3)
    \dd\bOmega \\
    = & \int_{\bbS^2} (\bOmega\cdot\bR_1)^2  
      (\bOmega\cdot\bR_2) w_1 f(\bOmega\cdot\bR_1; \gamma_1, \delta_1)
      \dd\bOmega \\
    = & 0,
\end{split}
  \]
  and
  \[
    \dfrac1{2\pi} \int_{\bbS^2} (\bOmega\cdot\bR_1)^2
    (\bOmega\cdot\bR_2)
    w_2 f(\bOmega\cdot\bR_2; \gamma_2, \delta_2)
    \dd\bOmega =  \frac 12 w_2 \int_{-1}^1
    (1 - \mu^2) \mu f(\mu;\gamma_2,\delta_2) \dd\mu.
  \]
Thus we have
\[
  \Vint{ (\bOmega\cdot\bR_1)^2 (\bOmega\cdot\bR_2)
  \bansatz} = \frac 12 w_2 \int_{-1}^1
  (1 - \mu^2) \mu f(\mu;\gamma_2,\delta_2) \dd\mu,
\]
and similarly,
\[
  \Vint{ (\bOmega\cdot\bR_3)^2 (\bOmega\cdot\bR_2)
  \bansatz} = \frac 12 w_2 \int_{-1}^1
  (1 - \mu^2) \mu f(\mu;\gamma_2,\delta_2) \dd\mu.
\]
Then we get
\[
  \Vint{ (\bOmega\cdot\bR_1)^2 (\bOmega\cdot\bR_2) \bansatz }=  
  \Vint{ (\bOmega\cdot\bR_3)^2 (\bOmega\cdot\bR_2) \bansatz }.
\]
On the other hand,
\[
  \Vint{ (\bOmega\cdot\bR_1)^2 (\bOmega\cdot\bR_2) \bansatz } +  
  \Vint{ (\bOmega\cdot\bR_3)^2 (\bOmega\cdot\bR_2) \bansatz } + 
  \Vint{ (\bOmega\cdot\bR_2)^3 \bansatz } = 
  \Vint{ (\bOmega\cdot\bR_2) \bansatz } = F_2.
\]
It follows that
\[
  \Vint{ (\bOmega\cdot\bR_1)^2 (\bOmega\cdot\bR_2) \bansatz } =   
  \Vint{ (\bOmega\cdot\bR_3)^2 (\bOmega\cdot\bR_2) \bansatz } = 
  \dfrac{ F_2 }{2}\left(1-
    \dfrac{ \sigma_2^2 
    + 2 F_2^2 - 3 w_2 \sigma_2 }
    { 2 F_2^2 - w_2 \sigma_2 - w_2^2 }
  \right).
\]

Now we look at $\Vint{ (\bOmega\cdot\bR_1) (\bOmega\cdot\bR_2)
(\bOmega\cdot\bR_3) \bansatz}$.
Once more by Lemma \ref{lem:integral-orthogonal},
\[
  \begin{split}
    & \int_{\bbS^2} (\bOmega\cdot\bR_1) (\bOmega\cdot\bR_2) 
    (\bOmega\cdot\bR_3) w_1 f(\bOmega\cdot\bR_1; 
    \gamma_1, \delta_1) \dd\bOmega \\
    = & \int_{\bbS^2} (\bOmega\cdot\bR_1) (\bOmega\cdot\bR_2) 
    (\bOmega\cdot\bR_3) w_2 f(\bOmega\cdot\bR_2; 
    \gamma_2, \delta_2) \dd\bOmega  \\
    = & \int_{\bbS^2} (\bOmega\cdot\bR_1) (\bOmega\cdot\bR_2) 
    (\bOmega\cdot\bR_3) w_3 f(\bOmega\cdot\bR_3; 
    \gamma_3, \delta_3) \dd\bOmega  \\
    = & ~0.
  \end{split}
\]
Therefore,
\[
  \Vint{ (\bOmega\cdot\bR_1) (\bOmega\cdot\bR_2)
  (\bOmega\cdot\bR_3) \bansatz} = 0.
\]
Summarizing the results from the above three cases completes the proof of this 
lemma.
\end{proof}

\tikzstyle{midpoint}=[circle,draw=black!50,fill=black,inner sep=1pt]
\begin{figure}[htbp]
  \centering
  \begin{tikzpicture}[font = \sansmath]
  \draw (0,0) coordinate (a)--+(3,0) coordinate (b)--+(60:3) coordinate (c) --cycle;
  \foreach\x[remember=\x as \lastx (initially c)] in{a,b,c}{
    \node[fill,circle,inner sep=1pt] at ($(\x)$) {};
  }
\node (pointP) at ( 1.2,1) [midpoint] {};
\node [black,above] at (pointP.east) {$P$};
\coordinate (Q) at (intersection of c--pointP and a--b);
\node (pointP1) at (Q) [midpoint] {};
\node [black,below] at (pointP1.east) {$P_1$};
\draw[dashed] (c) -- (Q);
\coordinate (U) at (intersection of a--pointP and b--c);
\node (pointP2) at (U) [midpoint] {};
\node [black,right] at (pointP2.east) {$P_2$};
\draw[dashed] (a) -- (U);
\coordinate (W) at (intersection of b--pointP and a--c);
\node (pointP3) at (W) [midpoint] {};
\node [black,left] at (pointP3.west) {$P_3$};
\draw[dashed] (b) -- (W);
\end{tikzpicture}
\caption{Schematic diagram of the interpolation}
\label{fig:schematic-interpolation}
\end{figure}
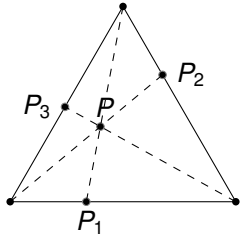

It remains to give $\sigma_i$, $i = 1,2,3$. Note that the trace of the
matrix $\bbrE^2$ equals $E^0$, so $\lambda_i$ satisfy the constraint
\[
\lambda_1 + \lambda_2 + \lambda_3 = E^0.
\]
And due to the positive semi-definiteness of $\bbrE^2$, we have
$\lambda_i\geq0$, $i=1,2,3$. This allows us to regard
$\left( \dfrac{\lambda_1}{E^0}, \dfrac{\lambda_2}{E^0},
  \dfrac{\lambda_3}{E^0} \right)$
as the barycentric coordinates of a point $\bP$ within a triangle (see
\figref{schematic-interpolation}). At the vertices of this triangle,
only one of the three eigenvalues of $\bbrE^2$ is non-zero. By the
similar arguments in the proof of Lemma \ref{lem:lambda2sigma}, a
non-negative $\bansatz$ in such cases retains only one of its three
terms. Combining this fact with \eqref{eq:w-cond} gives us the
closure at the vertices of the triangle:
\begin{center}
\begin{tabular}{ccc}
  $\left( \dfrac{\lambda_1}{E^0}, \dfrac{\lambda_2}{E^0},
  \dfrac{\lambda_3}{E^0} \right)$ & $\mapsto$ & 
  $(\sigma_1, \sigma_2, \sigma_3)$ \\
  $(1, 0, 0)$ & & $(E_0, 0, 0)$ \\
  $(0, 1, 0)$ & & $(0, E_0, 0)$ \\
  $(0, 0, 1)$ & & $(0, 0, E_0)$ \\
\end{tabular}
\end{center}
Now that the value of $(\sigma_1, \sigma_2, \sigma_3)$ at the vertices
are specified by the closure relation, we are to propose a smooth
extension of the functions $\sigma_i$ at the vertices to the whole
triangle, then a smooth extension of the closure relation is
achieved. A natural extension is a scaled identity map as
\[
\left( \dfrac{\lambda_1}{E^0}, \dfrac{\lambda_2}{E^0},
  \dfrac{\lambda_3}{E^0} \right) \mapsto \left( \lambda_1, \lambda_2,
  \lambda_3 \right),
\]
However, by \eqref{eq:w-cond} this extension results in
$w_j = \sigma_j$. As a consequence, the ansatz would always be linear
combinations of Dirac functions. It cannot include any smooth
functions, particularly it cannot recover a constant distribution at
the equilibrium. Moreover, such an extension does not depend on the
first-order moments $F_i$ at all, which is definitely not
appropriate. This motivates us to seek other ways of extending.

To figure out an appropriate extension, we assume it takes the following  
general but decomposed form
\begin{equation}\label{eq:sigma-decompose}
\sigma_i = \sum_{j=1}^3 s_j \sigma_i^j, \quad i = 1,2,3.
\end{equation}
It is assumed that $s_j$ is a weight function that relies only on $\lambda_j$, 
and $\sigma_i^j$ is a function that depends on both the first-order moments and the eigenvalues 
of the second-order moments but that is independent of $\lambda_j$. 

First, we determine the values of the weights, $s_j$. Our approach is 
motivated by geometric considerations. It is illustrated in
\figref{schematic-interpolation}. For the point $\bP$, we connect each 
vertex to $\bP$ and extend the line segment until it intersects with the
opposite side. Those three intersection points are denoted $\bP_j$,
$j = 1,2,3$, where the index $j$ indicates that $\bP_j$ lies on the side where
$\lambda_j = 0$. Denote the barycentric coordinates of $\bP_j$ by 
$\bP_j = \left( \dfrac{\lambda^j_1}{E^0}, \dfrac{\lambda^j_2}{E^0},
\dfrac{\lambda^j_3}{E^0} \right)$.
Therefore,
\[
  \lambda_i = \sum\limits_{j=1}^3 s_j \lambda^j_i,\quad
  j = 1,2,3,
\]
where
\begin{equation}\label{eq:weight-split} 
  s_1 = \frac12(\lambda_2+\lambda_3),\quad
  s_2 = \frac12(\lambda_1+\lambda_3),\quad
  s_3 = \frac12(\lambda_1+\lambda_2).
\end{equation}
The functions in \eqref{eq:weight-split} are used as the weights $s_j$, 
$j=1,2,3$.

The next thing is to specify $\sigma^j_i$.
Consider a $3 \times 3$ matrix with the nine functions,
$\sigma^j_i \left(\dfrac{\lambda_1}{E^0},
  \dfrac{\lambda_2}{E^0},\dfrac{\lambda_3}{E^0};
  \dfrac{F_1}{E^0},\dfrac{F_2}{E^0}, \dfrac{F_3}{E^0}\right)$,
$i, j = 1,2,3$, as its elements. Naturally, one would expect $\sigma^j_i$
to have symmetry in the permutation of indices. Precisely, if $\tau$ is a permutation on
the index set $\{1,2,3\}$, then for $\forall i, j = 1,2,3$,
\[
\sigma^j_i \left(\dfrac{\lambda_1}{E^0},
  \dfrac{\lambda_2}{E^0},\dfrac{\lambda_3}{E^0};
  \dfrac{F_1}{E^0},\dfrac{F_2}{E^0}, \dfrac{F_3}{E^0}\right) =
\sigma^{\tau(j)}_{\tau(i)} \left(\dfrac{\lambda_{\tau(1)}}{E^0},
  \dfrac{\lambda_{\tau(2)}}{E^0},\dfrac{\lambda_{\tau(3)}}{E^0};
  \dfrac{F_{\tau(1)}}{E^0},\dfrac{F_{\tau(2)}}{E^0},
  \dfrac{F_{\tau(3)}}{E^0}\right).
\]
Thus, we have only two functions for all $\sigma^j_i$:
\begin{itemize}
\item The three diagonal entries, $\sigma^i_i$, $i = 1,2,3$, have the
  same form;
\item All six off-diagonal entries, $\sigma^j_i$, $i \neq j$, have
  the same form.
\end{itemize}
Since $\sigma_i^j$ is assumed to be independent of $\lambda_j$, it
should be constant on the line segment ${\bP \bP_j}$. As an example, 
since $\sigma_i^1$ does not depend on $\lambda_1$, it should
be independent of $\lambda_2 + \lambda_3$. Therefore, one may use 
$\dfrac{\lambda_2} {\lambda_2 + \lambda_3}$ and
$\dfrac{\lambda_3}{\lambda_2 + \lambda_3}$ to replace $\lambda_2$ and $\lambda_3$
as variables in $\sigma_i^1$. Noticing that
$\left(0, \dfrac{\lambda_2} {\lambda_2 + \lambda_3},
\dfrac{\lambda_3}{\lambda_2 + \lambda_3}\right)$
is the barycentric
coordinate of $\bP_1$, we thus have
$\left. \sigma_i^1 \right|_{\bP} = \left. \sigma_i^1 \right|_{\bP_1}$,
and it is constant on line ${\bP \bP_1}$.

Moreover, this makes us assume $\sigma_i^j$ is also independent of $F_j$.
The reason is as follows. By Lemma \ref{lem:lambda2sigma}, the only region
in which \eqref{eq:B2-ansatz} might have a non-negative distribution when
$\lambda_j = 0$ is the rectangle $|F_k| \leq \lambda_k$, $k \neq j$.
Therefore, even when all three $\lambda_j$, $j = 1,2,3$, are positive,
we restrict our expected region to have a non-negative distribution 
inside the box $|F_k| \leq \lambda_k$, $k = 1,2,3$. Note that
this domain of $F_j$ depend on $\lambda_j$ while
$\sigma_i^j$ does not rely on $\lambda_j$, so we are induced to let 
$\sigma_i^j$ to be independent of $F_j$.

We proceed to specify $\sigma^j_i$ by constraints at vertices and
sides of the triangle. We first investigate the vertices to conclude
that
\begin{lemma}
With the assumptions above on $\sigma_i^j$, we have
\[
\sigma_i^i \equiv 0, \quad i = 1,2,3.
\]
\end{lemma}
\begin{proof}
  First, take the vertex in which $\lambda_1 = 1$ and
  $\lambda_2 = \lambda_3 = 0$. On this vertex one needs $\sigma_1 = 1$,
  and $\sigma_2 = \sigma_3 = 0$. We have
  \[
  \left.\sigma_1\right|_{\lambda_1 = 1} = \frac12(\left.\sigma^2_1\right|_{\lambda_1 = 1} 
  + \left.\sigma^3_1\right|_{\lambda_1 = 1}).
  \]
  Due to symmetry we know $\sigma^2_1 = \sigma^3_1$ on this vertex.
  Therefore, we have to let $\left.\sigma^2_1\right|_{\lambda_1 = 1}
  = \left.\sigma^3_1\right|_{\lambda_1 = 1} = 1$. Meanwhile, 
  \[
  \left.\sigma_2\right|_{\lambda_1 = 1} = \frac12(
  \left.\sigma^2_2\right|_{\lambda_1 = 1} + 
  \left.\sigma^3_2\right|_{\lambda_1 = 1}) = 0.
  \]
  This induces us to impose $\sigma^2_2 = \sigma^3_2 = 0$ on this vertex. 
  Next, consider the case on the side where $\lambda_1 = 0$. By Lemma
  \ref{lem:lambda2sigma}, $\sigma_1 = 0$. Recalling the consistency
  constraints \eqref{eq:w-cond}, we have
  \begin{equation}\label{eq:constraint-sigma23}
    \sigma_2 - \sigma_3 = \lambda_2 - \lambda_3.
  \end{equation}
  Consider any point $\bP$ on the side $\lambda_1 = 0$.
  Then, in \eqref{eq:sigma-decompose}, the function
  $\sigma^1_1$ takes its value at $\bP$ itself, while
  $\sigma^2_1$ is evaluated at the vertex $\lambda_3 = 1$, and 
  $\sigma^3_1$ is evaluated at the vertex $\lambda_2 = 1$.
  Then, on this side, we have
  \[
  \begin{split}
    \sigma_1 & = \frac12(\lambda_2 + \lambda_3) \sigma^1_1
    + \frac12(\lambda_1 + \lambda_3) \left.\sigma^2_1\right|_{\lambda_3 = 1}
    + \frac12(\lambda_1 + \lambda_2) \left.\sigma^3_1\right|_{\lambda_2 = 1} \\
    & = \frac12\sigma^1_1 + \frac12 \lambda_3 
    \left.\sigma^2_1\right|_{\lambda_3 = 1} + 
    \frac12 \lambda_2 \left.\sigma^3_1\right|_{\lambda_2 = 1} \\
    & = \frac12\sigma^1_1 = 0.
  \end{split}
  \]
  This proves that $\sigma^1_1 = 0$ on this side. 

  The above discussions show that $\sigma^i_i$ vanishes both at the vertex with 
  $\lambda_i = 1$ and on the side with $\lambda_i = 0$. Also, recall that $\sigma^i_i$ 
  is constant along straight lines passing through the vertex $\lambda_1 = 1$.
  Hence it is zero on the whole triangle. By symmetry, we have $\sigma_i^i
  \equiv 0$, $i=1,2,3$, on the whole triangle.

\end{proof}

We now turn to specifying $\sigma^j_i$ on the sides. 
On the side where $\lambda_1 = 0$, we also have
\[
  \begin{split}
    \sigma_2 = & \frac12(\lambda_2 + \lambda_3) \sigma^1_2 
    + \frac12(\lambda_1 + \lambda_3) 
    \left.\sigma^2_2\right|_{\lambda_3 = 1}
    + \frac12(\lambda_1 + \lambda_2) \left.\sigma^3_2\right|_{\lambda_2 = 1} \\
    = & \frac12 \sigma^1_2 + \frac12 \lambda_2 \left.\sigma^3_2\right|_{\lambda_2 = 1} 
    \qquad
  \text{notice } \sigma^3_2 = 1 \text{ on this vertex}\\
  = & \frac12 \sigma^1_2 + \frac12 \lambda_2,
\end{split}
\]
and
\[
  \sigma_3 = \frac12 \sigma^1_3 + \frac12 \lambda_3.
\]
Substracting these two equations yields
\[
  \sigma_2 - \sigma_3 = \frac12 (\sigma^1_2 - \sigma^1_3) + 
  \frac12 (\lambda_2 - \lambda_3).
\]
By \eqref{eq:constraint-sigma23}, we have
\[
\sigma^1_2 - \lambda_2 = \sigma^1_3 - \lambda_3.
\]
Recalling our previous assumption that $\sigma^1_2$ and $\sigma^1_3$ 
are independent of $\lambda_1$ and $F_1$, one has to set
\begin{equation}\label{eq:sigma-on-bound}
  \begin{split}
    & \sigma_2^1 = \lambda_2
    + h\left(\lambda_2, \lambda_3;
    F_2, F_3 \right),\\
    & \sigma_3^1 = \lambda_3
    + h\left(\lambda_2, \lambda_3;
    F_2, F_3\right),
  \end{split}
\end{equation}
where $h$ is a function with symmetry
\[
h(x,y;F_x,F_y) = h(y,x;F_y,F_x).
\]

The only thing remaining is to specify a particular
function $h$, so that all $\sigma_i^j$, $i \neq j$, would be assigned. In
choosing the function $h$, we have some constraints. For example:
\begin{enumerate}
\item On all three vertices, the values of $\sigma^j_k$ given by
  \eqref{eq:sigma-on-bound} are consistent with the discussions
  above.
\item The ansatz should cover the equilibrium distribution at the
  barycenter of the triangle.
\end{enumerate}
With these constraints, our objective is to find an $h$ for which the region
where $\bansatz$ is a non-negative integrable function is as large as 
possible. The requirements for $h$ can be summarized in the
following lemma:
\begin{lemma}\label{lem:requirements-h}
  Consider the case when $\lambda_1 = 0$.  For consistency with
  previous constraints on the vertices, the need to contain
  equilibrium, and to generate a non-negative ansatz for all moments
  within the region specified by Lemma \ref{lem:lambda2sigma}, $h$
  should satisfy the following:
  \begin{enumerate}
  \item \label{item:sign-h} $h(\lambda_2,\lambda_3; F_2, F_3) \leq 0$,
    within the rectangle $|F_j| \leq \lambda_j$, $j = 2,3$.
  \item \label{item:value-h} $-\frac12 h(\lambda_2,\lambda_3; F_2,F_3) \leq 
    \min\left\{\lambda_2 - |F_2|, \lambda_3 - |F_3|\right\}$.
  \item\label{item:vertex} $h(0,1; 0,F_y) = 0$, $h(1,0; F_x,0) = 0$.
  \item \label{item:equilibrium} $h(\frac12,\frac12,0,0) = -\frac13$.
  \item \label{item:boundary} $h(x, y; \pm x,\pm y) = 0$.
  \end{enumerate}
\end{lemma}
\begin{proof}
  Items \ref{item:sign-h} and \ref{item:value-h} come from requiring
  $\bansatz$ to be a non-negative distribution for the rectangle
  region in Lemma \ref{lem:lambda2sigma}.  Recalling that on the side
  $\lambda_1 = 0$, we have
  \begin{equation}\label{eq:sigma23}
    \sigma_j = \lambda_j + \frac12 h(\lambda_2,\lambda_3; F_2,F_3),
    \quad j = 2,3.
  \end{equation}
  From Lemma \ref{lem:lambda2sigma}, a non-negative distribution for
  \eqref{eq:B2-ansatz} in such cases require
  $|F_j| \leq \sigma_j \leq \lambda_j$, $j = 2,3$.  Hence
  \[
    h(\lambda_2,\lambda_3; F_2,F_3) \leq 0,
  \]
  and
  \[
    -\frac12 h(\lambda_2,\lambda_3; F_2,F_3) \leq 
    \min\left\{\lambda_2 - |F_2|, \lambda_3 - |F_3|\right\}.
  \]

  Item \ref{item:vertex} is due to consistency on vertices.  For
  instance, consider the case when $\lambda_2 = 1$, which should
  correspond to $\left.\sigma_2\right|_{\lambda_2 = 1} = 1$,
  $\left.\sigma_3\right|_{\lambda_2 = 1} = 0$. Plugging these into
  \eqref{eq:sigma23} gives item \ref{item:vertex}.

  Item \ref{item:equilibrium} comes from recovering equilibrium.  At
  equilibrium, $\lambda_j = \dfrac13$, $F_j = 0$, $j = 1,2,3$.  Direct
  calculation gives item \ref{item:equilibrium}.

  Item \ref{item:boundary} also derives from the non-negativity of
  the ansatz. It is a direct consequence of the discussions in Lemma
  \ref{lem:lambda2sigma}. In fact, it will naturally be satisfied if
  both requirements \ref{item:sign-h} and \ref{item:value-h} are
  satisfied. However, unlike either, it poses a direct constraint on the
  value of $h$ at certain points, which, therefore, is particularly
  useful when trying to propose a formula for $h$.

\end{proof}

In seeking $h(x,y; F_x,F_y)$, we start with item \ref{item:boundary}
in Lemma \ref{lem:requirements-h}, which suggests that $h(x,y; F_x,F_y)$
contains the factor
\begin{equation}\label{eq:h-factor}
  q(x,y; F_x,F_y) = \left(x - \dfrac{F_x^2}{x}\right) 
  \left(y - \dfrac{F_y^2}{y}\right).
\end{equation}
Note that as discussed in Lemma \ref{lem:lambda2sigma},
$\lambda_2 = 0$ would induce $F_2 = 0$, so this construction also
guarantees item \ref{item:vertex}. Also,
$q(\lambda_2,\lambda_3; F_2,F_3) \geq 0$ within the rectangle
$|F_j| \leq \lambda_j$, $j = 2,3$. Therefore, the remaining factor,
$ h(x,y; F_x, F_y) / q(x,y; F_x, F_y)$ is always
non-positive within $|F_j| \leq \lambda_j$, $j = 2,3$. We choose this
factor as a constant scaling of
\[
r(x,y;F_x,F_y) = -\left(1-\dfrac{F_x^2}{x} - \dfrac{F_y^2}{y}\right),
\]
which is always non-positive within the realizability domain.  The
constant factor is then given as $4/3$ based on item
\ref{item:equilibrium} in Lemma \ref{lem:requirements-h}. Therefore,
the function $h$ is set as
\begin{equation}\label{eq:h-formula}
  h(x,y;F_x,F_y) = \frac{4}{3} q(x,y;F_x,F_y) r(x,y;F_x,F_y).
\end{equation}
It is clear that it satisfies all items in Lemma
\ref{lem:requirements-h} except for item \ref{item:value-h}. The
precise depiction of the extent to which item \ref{item:value-h} is fulfilled
is deferred to the investigation of realizability in the next section.

With $h$ given, the whole model is closed. Direct
calculation gives us the closing relation of $\sigma_i$, $i = 1,2,3$, as
below:
\begin{equation}\label{eq:sigma-formula}
  \begin{split}
    & \sigma_1 = \lambda_1 - g(\lambda_1,\lambda_2; F_1, F_2)
    - g(\lambda_1,\lambda_3; F_1, F_3), \\
    & \sigma_2 = \lambda_2 - g(\lambda_2, \lambda_1; F_2, F_1)
    - g(\lambda_2,\lambda_3, F_2, F_3), \\
    & \sigma_3 = \lambda_3 - g(\lambda_3, \lambda_1; F_3, F_1)
    - g(\lambda_3, \lambda_2; F_3, F_2),
  \end{split}
\end{equation}
where
\[
g(x,y; F_x, F_y) = \dfrac{2 q(x,y;F_x,F_y) (x + y - 1 -
  r(x,y;F_x,F_y))} {3 (x + y)^2},
\]
satisfying $g(x,y; F_x,F_y) = g(y,x; F_y,F_x)$.

With $\sigma_j$ given as above, we substitute it into
\eqref{eq:w-cond} to give $w_i$, $i=1,2,3$, as
\begin{equation}\label{eq:w-formula}
  \begin{split}
    & w_1 = \sigma_1 + 2 g(\lambda_2, \lambda_3; F_2, F_3), \\
    & w_2 = \sigma_2 + 2 g(\lambda_1, \lambda_3; F_1, F_3), \\
    & w_3 = \sigma_3 + 2 g(\lambda_1, \lambda_2; F_1, F_2).
  \end{split}
\end{equation}
Then we plug $w_i$ and $\sigma_i$ into \eqref{eq:gamma-cond} to get
$\gamma_i$ and $\delta_i$. With formula for $w_i$, $\gamma_i$ and
$\delta_i$, $i = 1,2,3$, we now have the complete closed formula for
the ansatz $\bansatz$ in \eqref{eq:B2-ansatz}.

This closes our 3D $B_2$ model.  


%% file: analysis.tex

\section{Model Properties}\label{sec:properties}

In this section, we will study the rotational invariance, realizability,
and hyperbolicity of the 3D $B_2$ model proposed.

The proof of rotational invariance is almost straightforward for our
model. This is because all the parameters $\bR_i$, $w_i$,
$\gamma_i$, and $\delta_i$ in the ansatz $\bansatz$ are given as
functions of known moments $E^0$, $\bE^1$, and $\bbrE^2$. Consequently,
the ansatz is rotationally invariant, so we conclude that the
moment system produced by $\bansatz$ has rotational
invariance. More precisely, we have 
\begin{theorem}\label{lem:rotational_invariant}
  The 3D $B_2$ model \eqref{eq:moment_model} is rotationally invariant.
\end{theorem}
\begin{proof}
  For any unit vector $\bn = (n_x, n_y, n_z) \in \mathbb{R}^3$, there
  exists a rotation to transform $\bn$ to the $x$-axis. Let 
  $\bn = \bbrT \be_x$, where $\bbrT$ is the rotation matrix. The
  rotated velocity is denoted by $\tilde{\bOmega} = \bbrT^T \bOmega$.
  We denote $\tilde{\be}_x = \bn = \bbrT \be_x$,
  $\tilde{\be}_y = \bbrT \be_y$ and $\tilde{\be}_z = \bbrT \be_z$.
  After the rotation, the known moments are denoted by $\tilde{\bE}$,
  and we write the ansatz before and after the rotation with
  explicit dependence on the known moments by
  $\bansatz(\bOmega; \bE)$ and $\bansatz(\bOmega; \tilde{\bE})$. We
  use $\tilde{E}^0$, $\tilde{\bE}^1$, and $\tilde{\bbrE}^2$ to denote
  the corresponding moments after the rotation, respectively. Let us
  define $\tilde{\bv}$ as
  \[
  \begin{array}{rllll}
    \tilde{\bv} = [ & 1, &&& \\ [2mm]
                    & (\bOmega\cdot\tilde{\be}_x), 
                         & (\bOmega\cdot\tilde{\be}_y), 
                           & (\bOmega\cdot\tilde{\be}_z), & \\ [2mm]
                    & (\bOmega\cdot\tilde{\be}_x)^2, 
                         & (\bOmega\cdot\tilde{\be}_x)(\bOmega\cdot\tilde{\be}_y), 
                           & (\bOmega\cdot\tilde{\be}_x)(\bOmega\cdot\tilde{\be}_z), & \\
                    && (\bOmega\cdot\tilde{\be}_y)^2, 
                         &(\bOmega\cdot\tilde{\be}_y)(\bOmega\cdot\tilde{\be}_z) & ]^T.
  \end{array}
  \]
  It is clear there exists a transformation matrix $\mathbb{T}$
  which depends only on $\bbrT$ such that
  \[
    \tilde{\bv} = \mathbb{T} \bv,
  \]
  where $\bv$ is defined in \eqref{eq:v-def}.
  Thus, the known moments satisfy $\tilde{\bE} = \mathbb{T} \bE$ and
  \[ 
  \tilde{E}^0 = E^0, \quad \tilde{\bE}^1 = \bbrT \bE^1, \quad
  \tilde{\bbrE}^2 = \bbrT^T \bbrE^2 \bbrT.
  \]
  Consequently, the eigenvectors $\tilde{\bR}_i$ of $\tilde{\bbrE}^2$
  are $\tilde{\bR}_i = \bbrT^T \bR_i$, and thus,
  \[
  \begin{aligned}
    &\tilde{\bOmega} \cdot \tilde{\bR}_i = \bOmega \cdot \bR_i, \\
    &\tilde{F}_i = \tilde{\bE}^1 \cdot \tilde{\bR}_i = \bE^1 \cdot
    \bR_i = F_i.
  \end{aligned}
  \]
  The given closure for $w_i$, $\gamma_i$, and $\delta_i$
  are functions of the eigenvalues of $\bbrE^2$ and $F_i$, $i=1,2,3$. Thus,
  these parameters are exactly the same before and after the
  rotation. Therefore, the ansatz after the rotation satisfies 
  \[
  \begin{aligned}
    \bansatz(\tilde{\bOmega}; \mathbb{T} \bE) &= \sum_{i=1}^3
    \dfrac1{2\pi} w_i f(\tilde{\bOmega} \cdot \tilde{\bR}_i; \gamma_i,
    \delta_i) \\
    &= \sum_{i=1}^3 \dfrac1{2\pi} w_i f(\bOmega \cdot \bR_i; \gamma_i,
    \delta_i) = \bansatz(\bOmega; \bE).
  \end{aligned}
  \]
  Meanwhile, notice that we have the relation
  \[
  \begin{array}{rllll}
    \tilde{\bv} = [ & 1, &&& \\ [2mm]
                    & (\tilde{\bOmega}\cdot \be_x), 
                         & (\tilde{\bOmega}\cdot \be_y), 
                           & (\tilde{\bOmega}\cdot \be_z), & \\ [2mm]
                    & (\tilde{\bOmega}\cdot\be_x)^2, 
                         & (\tilde{\bOmega}\cdot\be_x)(\tilde{\bOmega}\cdot\be_y), 
                           & (\tilde{\bOmega}\cdot\be_x)(\tilde{\bOmega}\cdot\be_z), & \\
                    && (\tilde{\bOmega}\cdot\be_y)^2, 
                         &(\tilde{\bOmega}\cdot\be_y)(\tilde{\bOmega}\cdot\be_z)
                           & ]^T \\
  \end{array}
  \]
  Therefore,
  \[
  \begin{split}
    n_x \bff_x(\bE) + n_y \bff_y(\bE) + n_z \bff_z(\bE) = &
    \int_{\bbS^2}(n_x \Omega_x + n_y \Omega_y + n_z \Omega_z)
    \bv \bansatz (\bOmega; \bE) \dd\bOmega \\
    = & \int_{\bbS^2}(\bOmega \cdot \tilde{\be}_x)~ (\mathbb{T}^{-1}
    \tilde{\bv})~
    \bansatz(\tilde{\bOmega}; \mathbb{T} \bE) ~ | \bbrT | ~\dd \tilde{\bOmega} \\
    = & ~\mathbb{T}^{-1} \int_{\bbS^2} (\tilde{\bOmega} \cdot \be_x)~
    \tilde{\bv}
    \bansatz(\tilde{\bOmega}; \mathbb{T} \bE) \dd \tilde{\bOmega} \\
    = & ~\mathbb{T}^{-1} \int_{\bbS^2} (\bOmega \cdot \be_x)~ \bv
    \bansatz(\bOmega; \mathbb{T} \bE) \dd \bOmega \\
    = & ~\mathbb{T}^{-1} \bff_x(\mathbb{T} \bE).
  \end{split}
  \]
  This gives us rotational invariance \footnote{We note that the
    proof is not at all dependent on whether the function $f$ in the ansatz is
    assigned as a beta distribution.}.
\end{proof}

Let us turn to the realizability of our model. First, we point out
that the 3D $B_2$ model provides a non-negative ansatz even for some
moments on the boundary of the realizability domain. For example,
the moments satisfying $|F_i| = \lambda_i$, $\forall i = 1, 2, 3$, 
correspond to ans\"atze of the form
\[
  \bansatz = \sum\limits_{i=1}^3
  \left[\alpha^+_i
    \delta(\bOmega\cdot\bR_i-1)+
    \alpha^-_i
  \delta(\bOmega\cdot\bR_i+1)\right].
\]
We recall the following results from Lemma \ref{lem:realizable-bound}: if
$\lambda_i$ are distinct positive values, then the eight vertices of 
the rectangular box $|F_j| \leq \lambda_j$, $j = 1,2,3$, are the only points 
on the boundary of the realizability domain where a non-negative 
ansatz for $\bansatz$ may exist.
Moreover, the ansatz contains the equilibrium distribution.
Moments satisfying $\lambda_i = \dfrac{E^0}{3}$, $i=1,2,3$, and
$\bE^1 = \boldsymbol{0}$ reproduce $\bansatz = \dfrac{E^0}{4\pi}$.

Recall that 
\eqref{eq:condition-positive-ansatz} is a sufficient condition for
\eqref{eq:B2-ansatz} to give a non-negative ansatz. It is equivalent
to
\begin{equation}\label{eq:condition-realizable-v1}
  0 \leq \sigma_i \leq w_i,\quad
  \text{and}\quad \sigma_i w_i \geq F_i^2,\quad
  i = 1,2,3.
\end{equation}
We examine this condition to check the realizability of our
model. Define the following discriminant
\begin{equation}\label{eq:discriminant}
  \Delta \triangleq \min \left\{w_1 \sigma_1 - F_1^2, w_2 \sigma_2 - F_2^2, w_3
  \sigma_3 - F_3^2 \right\}.
\end{equation}
Instantly, we have
\begin{theorem}\label{thm:condition-nonnegative-ansatz}
  For $|F_j| \leq \lambda_j \neq 0$, $j = 1,2,3$, the 3D $B_2$ model has
  a non-negative ansatz $\bansatz$ if 
  \begin{equation}\label{eq:cond-sigma-positive}
    3 \lambda_i^2 + \lambda_i (\lambda_j + \lambda_k) - \lambda_j
    \lambda_k > 0, \quad \forall~i,j,k,~~\text{mutually different},
  \end{equation}
  and
  \[ 
  \Delta \geq 0.
  \]
\end{theorem}
\begin{proof}
  We first prove $\sigma_1 \leq w_1$. Notice that 
  \[
    w_1 - \sigma_1 = 2 g(\lambda_2,\lambda_3; F_2,F_3)
    = \dfrac{4 \left(\lambda_2 - \dfrac{F_2^2}{\lambda_2}\right)
      \left(\lambda_3-\dfrac{F_3^2}{\lambda_3}\right)
      \left(\lambda_2 - \dfrac{F_2^2}{\lambda_2} +
    \lambda_3 - \dfrac{F_3^2}{\lambda_3}\right)}{3
    (\lambda_2 + \lambda_3)^2}.
  \]
  Also, if $|F_j| \leq \lambda_j$, we have 
  \[
    \lambda_j - \dfrac{F_j^2}{\lambda_j} \geq 0.
  \]
  Therefore, inside the rectangular box $|F_j| \leq \lambda_j$, $j = 1,2,3$,
  we have $w_1 - \sigma_1 \geq 0$. Similarly, we could prove
  $\sigma_2 \leq w_2$ and $\sigma_3 \leq w_3$.

  We now discuss the condition for $\sigma_i \geq 0$, $i = 1,2,3$.
  We begin by examining $\sigma_1$. From \eqref{eq:sigma-formula}, we see that 
  for fixed $\lambda_i$, $i = 1,2,3$, the function $\sigma_1$ monotonically increases
  for any $|F_j|$. Therefore, if $\sigma_1 \geq 0$ holds for $\bE^1 = \boldsymbol{0}$, 
  then it is valid for the whole rectangular box
  $|F_j| \leq \lambda_j$, $j = 1,2,3$. So, the problem becomes seeking
  $(\lambda_1, \lambda_2,\lambda_3)$ for which
  $\left.\sigma_1\right|_{F_1 = F_2 = F_3 = 0} \geq 0$ holds. As
  \[
    \left.\sigma_1\right|_{F_1=F_2=F_3=0} = 
    \dfrac{\lambda_1(3\lambda_1^2 + \lambda_1\lambda_2 + \lambda_1 \lambda_3
    - \lambda_2\lambda_3)}{3(\lambda_1 + \lambda_2)(\lambda_1 + \lambda_3)},
  \]
  the necessary and sufficient condition for $\sigma_1 > 0$ is
  \begin{equation}\label{eq:cond-sigma-positive-1}
    3\lambda_1^2 + \lambda_1(\lambda_2 + \lambda_3) - \lambda_2\lambda_3 
    > 0,
  \end{equation}
which completes our proof.
\end{proof}
From the proof of Theorem \ref{thm:condition-nonnegative-ansatz}
we have the following corollary.
\begin{corollary}
  Let $\bE^1 = \boldsymbol{0}$.  If \eqref{eq:cond-sigma-positive} is
  valid and $\lambda_i \neq 0$, $\forall i = 1,2,3$, the 3D $B_2$
  model has a non-negative ansatz.
\end{corollary}
\begin{proof}
  In the case of $\bE^1 = \boldsymbol{0}$, $\Delta > 0$ is automatically
  valid under the conditions specified in the corollary.
\end{proof}

Given $\lambda_i$ and $F_i$, $i=1,2,3$, we could use the 
condition placed on the discriminant $\Delta$ in Theorem 
\ref{thm:condition-nonnegative-ansatz} to verify whether 
a non-negative ansatz exists. For each fixed 
$(\lambda_1, \lambda_2, \lambda_3)$, we sample for  
the whole region within the rectangular box $|F_j| \leq \lambda_j$, 
$j = 1,2,3$. It is found that if 
$\dfrac{\lambda_i}{E^0}\geq\dfrac17$, $i = 1,2,3$, then for any
$(F_1, F_2, F_3)$ belonging to the region $|F_j| \leq \lambda_j$,
$j = 1,2,3$, the 3D $B_2$ model has a non-negative ansatz. Note
that the realizability domain for $F_j$ is the ellipsoid given in Lemma
\ref{proposition-realizability}, and the rectangular box $|F_j| \leq \lambda_j$,
$j = 1,2,3$, is contained within the ellipsoid, with its eight
vertices touching the domain boundary. 
Figure \ref{fig:realizable-region} illustrates the region that is found to
admit a non-negative ansatz.
\begin{figure}
  \subfigure[$\left(\dfrac{\lambda_1}{E^0},\dfrac{\lambda_2}{E^0},
    \dfrac{\lambda_3}{E^0}\right)$ are taken as barycentric coordinates within the 
    triangle. The outer triangle is the realizability domain. The curves
    correspond to the outer boundary of the constraints \eqref{eq:cond-sigma-positive}.
    The blue region gives non-negative ansatz for 3D $B_2$ model
    for all $\bE^1$ satisfying $|F_j|\leq\lambda_j$, $j = 1,2,3$. 
  ]{
  \includegraphics[width=0.4\textwidth]{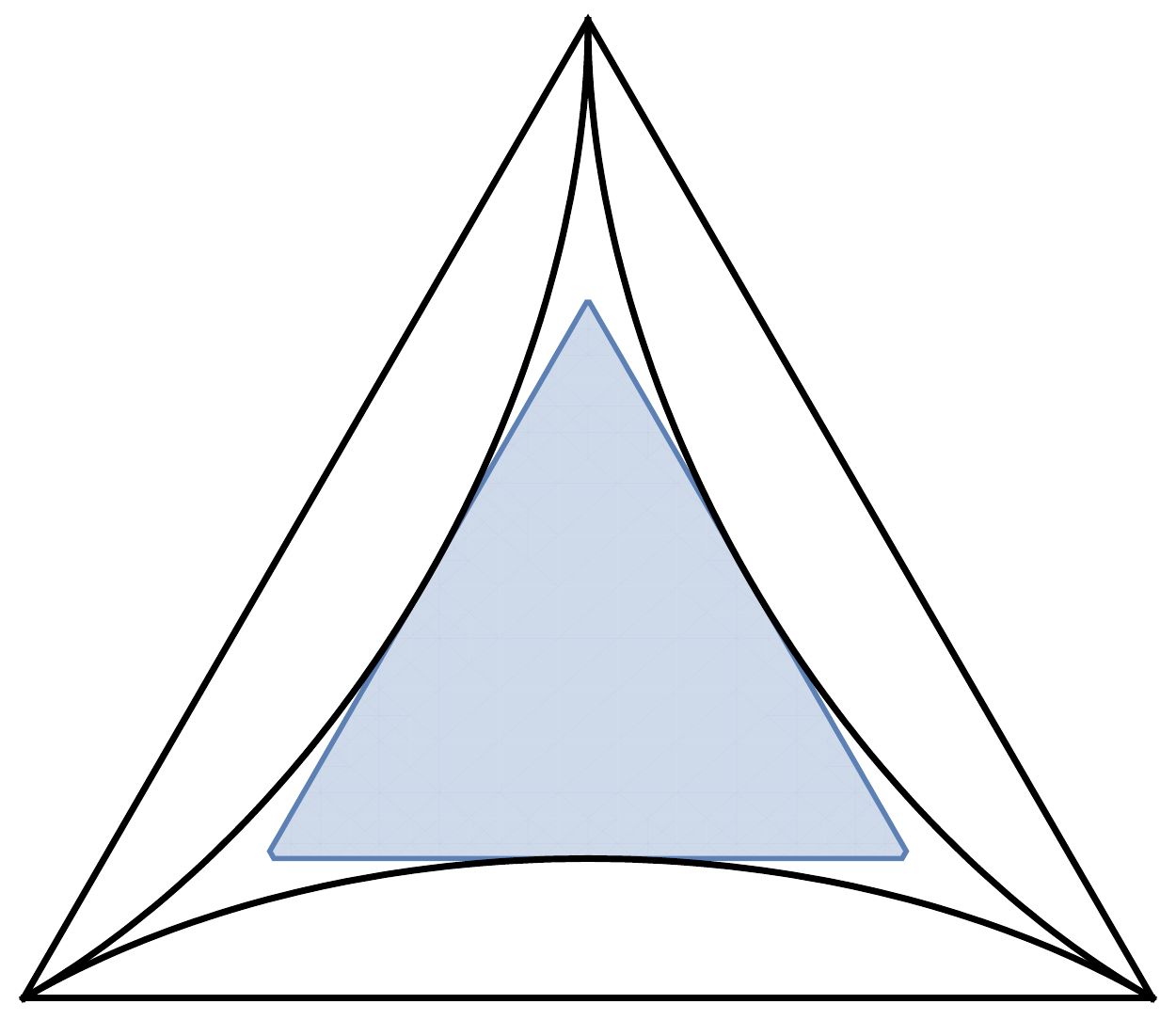}}  
  \hfill
  \subfigure[The sphere correspond to the realizability domain of $\bE^1$ when
    $\lambda_1 = \lambda_2 = \lambda_3$.The rectangle within the sphere is
  the region for $\bE^1$ when the 3D $B_2$ model has a non-negative ansatz.]{
  \includegraphics[width=0.35\textwidth]{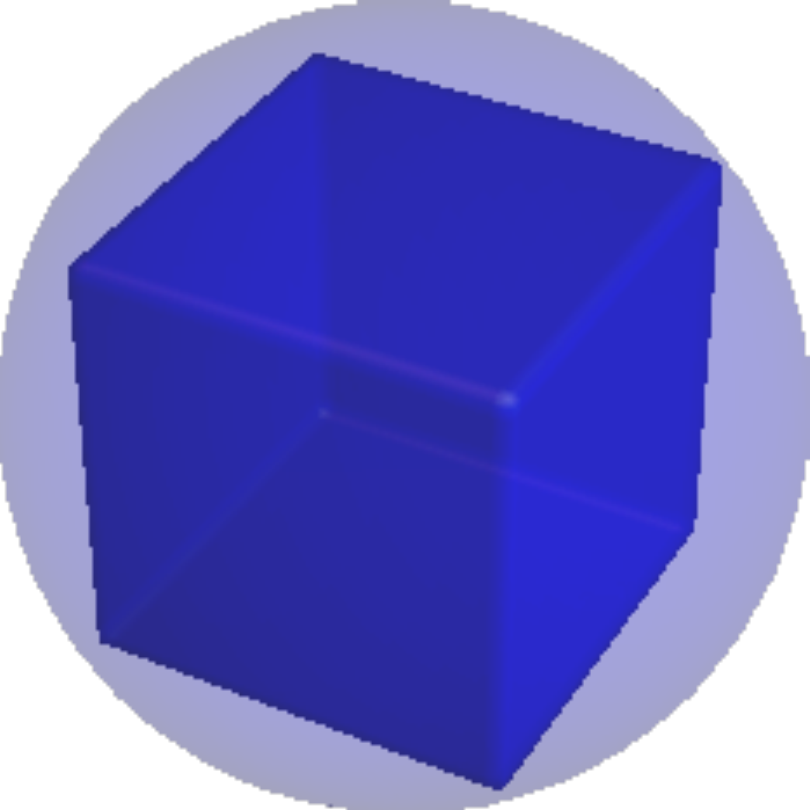}} 
  \hfill
  \caption{Region which correspond to a non-negative ansatz for 3D $B_2$ model.}
  \label{fig:realizable-region}
\end{figure}

\begin{remark}
  By Lemma \ref{lem:E3}, for $\bE^1=\boldsymbol{0}$, the third-order
  moments given by the 3D $B_2$ ansatz is a zero tensor, equal to that given by
  $M_2$. For this particular case, even when there is no non-negative ansatz, 
  the closure relation is still realizable.
\end{remark}

We proceed to study the hyperbolicity of the model. Due to the extreme
complexity of the formula, we restrict our discussions to the case that
$\bE^1 = \boldsymbol{0}$. We first prove the following facts:
\begin{lemma}\label{lem:fact-bound-sigma-w}
  In the interior of the realizability domain $\cM$, if
  $\bE^1 = \boldsymbol{0}$, we have
  \[
  w_i > 0,\quad \sigma_i + w_i > 0,\quad i = 1,2,3.
  \]
\end{lemma}
\begin{proof}
  Take $i = 1$ for example. First, note
  \[
    g(x,y; 0,0) = \dfrac{2 q(x,y; 0,0)(x + y -1 -r(x,y;0,0))}{3(x+y)^2} 
    = \dfrac{2 xy}{3(x+y)}.
  \]
  Therefore,
  \[
    \begin{split}
      w_1 =    & \lambda_1 - g(\lambda_1,\lambda_2;0,0) -g(\lambda_1, \lambda_3;0,0)
      + 2 g(\lambda_2,\lambda_3;0,0) \\
      = & \dfrac13\left(3\lambda_1 - \dfrac{2\lambda_1 \lambda_2}
      {\lambda_1 + \lambda_2} - \dfrac{2\lambda_1 \lambda_3}
      {\lambda_1 + \lambda_3} + \dfrac{4 \lambda_2 \lambda_3}
    {\lambda_2 + \lambda_3} \right) \\
    = & \dfrac13\left(\lambda_1 - \dfrac{2\lambda_1 \lambda_2}
    {\lambda_1 + \lambda_2} + 2 \lambda_1 - \dfrac{2 \lambda_1 \lambda_3}
    {\lambda_1 + \lambda_3}
  + \dfrac{4 \lambda_2 \lambda_3}{\lambda_2 + \lambda_3}\right) \\
  = & \dfrac13\left(\lambda_1 - \dfrac{2\lambda_1 \lambda_2}
  {\lambda_1 + \lambda_2} + \dfrac{2 \lambda_1(\lambda_1 + \lambda_3 - \lambda_3)}
  {\lambda_1 + \lambda_3}
+ \dfrac{4 \lambda_2 \lambda_3}{\lambda_2 + \lambda_3}\right) \\
= & \dfrac13\left(\lambda_1 - \dfrac{2\lambda_1 \lambda_2}
{\lambda_1 + \lambda_2} + \dfrac{2 \lambda_1^2}{\lambda_1 + \lambda_3}
    + \dfrac{4 \lambda_2 \lambda_3}{\lambda_2 + \lambda_3}\right). 
  \end{split}
\]
We need to prove $w_1 \geq 0$ for two cases: 
  \begin{enumerate}
  \item $\lambda_1 \geq \lambda_2$ or $\lambda_1 \geq \lambda_3$.
    Because $w_1$ is symmetric with respect to $\lambda_2$ and
    $\lambda_3$, we only need to discuss the case 
    $\lambda_1 \geq \lambda_2$.
  \item $\lambda_1 < \lambda_2$ and $\lambda_1 < \lambda_3$.
    Due to $w_1$ being symmetric about $\lambda_2$ and $\lambda_3$,
    we only need to discuss the case $\lambda_1 < \lambda_2 \leq
    \lambda_3$.
  \end{enumerate}
  The proof is as follows:
  \begin{enumerate}
  \item If $\lambda_1 \geq \lambda_2$, then
    \[
    -\dfrac{2\lambda_1 \lambda_2}{\lambda_1 + \lambda_2}
    \geq -\dfrac{2\lambda_1 \lambda_2}{\lambda_2 + \lambda_2}
    = -\lambda_1,
    \]
    therefore
    \[
    w_1 = \dfrac13\left(\lambda_1 - \dfrac{2\lambda_1 \lambda_2}
      {\lambda_1 + \lambda_2} + \dfrac{2 \lambda_1^2}{\lambda_1 + \lambda_3}
      + \dfrac{4 \lambda_2 \lambda_3}{\lambda_2 + \lambda_3}\right)
    \geq \dfrac13\left(\dfrac{2 \lambda_1^2}{\lambda_1 + \lambda_3}
      + \dfrac{4 \lambda_2 \lambda_3}{\lambda_2 + \lambda_3}\right) > 0.
    \]
  \item If $\lambda_1 < \lambda_2 \leq \lambda_3$, then
    \[
    -\dfrac{2\lambda_1 \lambda_2}{\lambda_1 + \lambda_2}
    \geq -\dfrac{2\lambda_1 \lambda_2}{\lambda_1 + \lambda_1}
    = -\lambda_2,
    \]
    and
    \[
    \dfrac{4\lambda_2 \lambda_3}{\lambda_2 + \lambda_3} \geq
    \dfrac{4\lambda_2 \lambda_3}{\lambda_3 + \lambda_3}
    = 2\lambda_2.
    \]
    Therefore
    \[
    w_1 = \dfrac13\left(\lambda_1 - \dfrac{2\lambda_1 \lambda_2}
      {\lambda_1 + \lambda_2} + \dfrac{2 \lambda_1^2}{\lambda_1 + \lambda_3}
      + \dfrac{4 \lambda_2 \lambda_3}{\lambda_2 + \lambda_3}\right)
    \geq \dfrac13\left(\lambda_1 - \lambda_2 + 
      \dfrac{2 \lambda_1^2}{\lambda_1 + \lambda_3} + 2\lambda_2
    \right)
    > 0.
    \]
  \end{enumerate}
  This proves $w_1 > 0$. Similarly, $w_j > 0$, $j = 2,3$.

  Next, we prove $\sigma_1 + w_1 > 0$. We have
  \[
    \sigma_1 + w_1   = \dfrac23\left(3\lambda_1 - \dfrac{2\lambda_1 \lambda_2}
      {\lambda_1 + \lambda_2} - \dfrac{2\lambda_1 \lambda_3}
      {\lambda_1 + \lambda_3} + \dfrac{2 \lambda_2 \lambda_3}
    {\lambda_2 + \lambda_3} \right)
    = \dfrac23\left(\lambda_1 - 
      \dfrac{2\lambda_1 \lambda_2}{\lambda_1 + \lambda_2}
      + \dfrac{2\lambda_1^2}{\lambda_1 + \lambda_3} + 
      \dfrac{2\lambda_2 \lambda_3}{\lambda_2 + \lambda_3}
    \right).
  \]
  Similar to discussions on $w_1$, we have
  \begin{enumerate}
  \item If $\lambda_1 \geq \lambda_2$, then
    \[
    \sigma_1 + w_1 =   \dfrac23\left(\lambda_1 - 
      \dfrac{2\lambda_1 \lambda_2}{\lambda_1 + \lambda_2}
      + \dfrac{2\lambda_1^2}{\lambda_1 + \lambda_3} + 
      \dfrac{2\lambda_2 \lambda_3}{\lambda_2 + \lambda_3}
    \right)
    \geq 
    \dfrac23\left(\dfrac{2\lambda_1^2}{\lambda_1 + \lambda_3} + 
      \dfrac{2\lambda_2 \lambda_3}{\lambda_2 + \lambda_3}
    \right).
    \]
  \item If $\lambda_1 < \lambda_2 \leq \lambda_3$, then
    \[
    \sigma_1 + w_1 = \dfrac23\left(\lambda_1 - 
      \dfrac{2\lambda_1 \lambda_2}{\lambda_1 + \lambda_2}
      + \dfrac{2\lambda_1^2}{\lambda_1 + \lambda_3} + 
      \dfrac{2\lambda_2 \lambda_3}{\lambda_2 + \lambda_3}
    \right)
    \geq \dfrac23\left(\lambda_1 - \lambda_2 + 
      \dfrac{2 \lambda_1^2}{\lambda_1 + \lambda_3} + \lambda_2
    \right)
    > 0.
    \]
  \end{enumerate}
  Similarly, $\sigma_i + w_i > 0$, $i = 2,3$.
\end{proof}

To study hyperbolicity, we start with calculating the Jacobian
matrix of the flux $\bff_x$, $\bff_y$, and $\bff_z$. Due to the rotational
invariance of the 3D $B_2$ model, it could be assumed without loss of
generality that $\bbrE^2$ is diagonal, $\bR_1$ is parallel to
the $x$-axis, $\bR_2$ is parallel to the $y$-axis, and $\bR_3$ is
parallel to the $z$-axis, respectively. The most involving part in
calculating the Jacobian matrix is the derivatives of third-order
moments. We first note that, by Lemma \ref{lem:E3}, fixing
$\bE^1 = \boldsymbol{0}$ makes the value of all third-order moments
zero, no matter what the values of the other moments are. Therefore,
\[
\pd{E^3_{ijk}}{E^0} = 0, \quad \pd{E^3_{ijk}}{E^2_{lm}} = 0, 
\quad\forall i,j,k,l,m = 1,2,3.
\]
So we only need to compute $\pd{E^3_{ijk}}{E^1_l}$.
First, we have
\[
\begin{split}
  \pd{E^3_{123}}{E^1_l} = & \pd{\Vint{(\bOmega\cdot\bR_i) (\bOmega\cdot\bR_j)
      (\bOmega\cdot\bR_k) \bansatz}}{E^1_l} R_{i1} R_{j2} R_{k3}\\
  = & \pd{\Vint{(\bOmega\cdot\bR_1) (\bOmega\cdot\bR_2) (\bOmega\cdot\bR_3) \bansatz}}{E^1_l} 
  =  0.
\end{split}
\]
For the terms $\pd{E^3_{iij}}{E^1_k}$, we have
\[
\pd{E^3_{iij}}{E^1_k} = \pd{\Vint{(\bOmega\cdot\bR_l) (\bOmega\cdot\bR_m)
    (\bOmega\cdot\bR_n) \bansatz}}{E^1_k} R_{li} R_{mi} R_{nj} 
= \pd{\Vint{(\bOmega\cdot\bR_i)^2 (\bOmega\cdot\bR_j) \bansatz}}{E^1_k}.
\]
And by
\[
F_i = \bE^1\cdot\bR_i = E^1_1 R_{1i} + E^1_2 R_{2i} + E^1_3 R_{3i},
\]
we get $\pd{F_i}{E^1_k} = \delta_{ik}$, which is used below in
computing $\pd{E^3_{iij}}{E^1_k}$.

If $i = j$ and $k \not= i$,
\[
\pd{\Vint{(\bOmega\cdot\bR_i)^3 \bansatz}}{E^1_k} 
=  F_i \dfrac{\partial}{\partial E^1_k}\left( \dfrac{ \sigma_i^2 
    + 2 F_i^2 - 3 w_i \sigma_i }
  {2 F_i^2 - w_i \sigma_i - w_i^2}\right) 
= 0.
\]
And if $i \not= j$ and $k \not= j$,
\[
\pd{\Vint{(\bOmega\cdot\bR_i)^2 (\bOmega\cdot\bR_j)
    \bansatz}}{E^1_k} 
=  \dfrac{ F_j }{2}\dfrac{\partial}{\partial E^1_k} 
\left(1-\dfrac{ \sigma_j^2 
    + 2 F_j^2 - 3 w_j \sigma_j }
  { 2 F_j^2 - w_j \sigma_j - w_j^2 }
\right)
=  0.
\]
Therefore, the non-zero entries in the Jacobian matrix can be
$\pd{E^3_{iij}}{E^1_j}$ only. By rotational invariance of the model,
we need only study the Jacobian matrix in the $x$-direction,
$\pd{\bff_x}{\bE}$, which is
\begin{equation}\label{eq:Jacobi-of-ApproxM2-sx}
  \renewcommand\arraystretch{2.}
  \boldsymbol{\mathrm{J}}_x = 
  \left(
    \begin{array}{ccccccccc}
      0 & 1 & 0 & 0 & 0 & 0 & 0 & 0 & 0 \\
      0 & 0 & 0 & 0 & 1 & 0 & 0 & 0 & 0 \\
      0 & 0 & 0 & 0 & 0 & 1 & 0 & 0 & 0 \\
      0 & 0 & 0 & 0 & 0 & 0 & 1 & 0 & 0 \\
      0 & \pd{E^3_{111}}{E^1_1} & 0 & 
                                      0 & 0 & 0 & 0 & 0 & 0\\
      0 & 0 & \pd{E^3_{112}}{E^1_2} & 
                                      0 & 0 & 0 & 0 & 0 & 0\\
      0 & 0 & 0 & 
                  \pd{E^3_{113}}{E^1_3} & 0 & 0 & 0 & 0 & 0\\
      0 & \pd{E^3_{122}}{E^1_1} & 0 & 
                                      0 & 0 & 0 & 0 & 0 & 0\\
      0 & 0 & 0 & 0 & 0 & 0 & 0 & 0 & 0
    \end{array}
  \right).
\end{equation}
For the non-zero entries in $\boldsymbol{\mathrm{J}}_x$, we have the
following bounds:
\begin{lemma}\label{lem:E3_11k-bound}
  In the interior of the realizability domain $\cM$, if
  $\bE^1 = \boldsymbol{0}$, we have
  \begin{enumerate}
    \item $0 < \pd{E^3_{11k}}{E^1_k} < \dfrac12$, for $k = 2,3$;
    \item $0 < \pd{E^3_{111}}{E^1_1} \leq 1$ if and only if $\sigma_1 > 0$.
  \end{enumerate}
\end{lemma}
\begin{proof}
  For the first item, we only need to verify for $k=2$. By Lemma \ref{lem:E3},
  one has
  \[
  \begin{split}
    \pd{E^3_{112}}{E^1_2} = & \pd{\Vint{(\bOmega\cdot\bR_1)^2
        (\bOmega\cdot\bR_2) \bansatz}}{E^1_2} \\ = & \dfrac12\left(1 -
      \dfrac{\sigma_2^2 + 2 F_2^2 - 3 w_2 \sigma_2}{2 F_2^2 - w_2
        \sigma_2 -w_2^2}\right) \\
    = & \dfrac12 \dfrac{(w_2 - \sigma_2)^2}{w_2 (\sigma_2 + w_2)}.
  \end{split}
  \]
  By Lemma \ref{lem:fact-bound-sigma-w} we have $w_2 > 0$ and 
  $\sigma_2 + w_2 > 0$, thus, $\pd{E^3_{112}}{E^1_2} > 0$. 
  In addition, from the proof of Theorem \ref{thm:condition-nonnegative-ansatz},
  we have $\sigma_2 \leq w_2$, therefore $\pd{E^3_{112}}{E^1_2} < \dfrac12$.

  For the second item, we have by Lemma \ref{lem:E3},
 \[
 \begin{aligned}
   \pd{E^3_{111}}{E^1_1} = & \pd{\Vint{(\bOmega\cdot\bR_1)^3
       \bansatz}}{E^1_1} 
   = \dfrac{\sigma_1^2 + 2 F_1^2 - 3 w_1 \sigma_1}{2 F_1^2 - w_1
     \sigma_1 -w_1^2} \\
   = &\dfrac{\sigma_1 (3 w_1 - \sigma_1)}{w_1 (\sigma_1 + w_1)}
   = 1 - \dfrac{(w_1 - \sigma_1)^2}{w_1 (\sigma_1 + w_1)} \leq 1,
 \end{aligned}
 \]
 And $\pd{E^3_{111}}{E^1_1} > 0$ is equivalent to
 $\sigma_1 (3 w_1 - \sigma_1) > 0$. As $\sigma_1 \leq w_1$,
 we have $3 w_1 - \sigma_1 \geq 2 w_1 > 0$, implying that
 $\pd{E^3_{111}}{E^1_1} > 0$ is equivalent to $\sigma_1 > 0$.
\end{proof}
We now give the condition for the real diagonalizability of the Jacobian
matrix $\boldsymbol{\mathrm{J}}_x$ as follows:
\begin{theorem}\label{lem:Jx-invertible}
  The Jacobian matrix $\boldsymbol{\mathrm{J}}_x$ defined in
  \eqref{eq:Jacobi-of-ApproxM2-sx} is real diagonalizable
  if and only if $\sigma_1 > 0$.
\end{theorem}
\begin{proof}
The characteristic polynomial of $\boldsymbol{\mathrm{J}}_x$ is
\begin{equation}
  \left| \lambda \bI - \boldsymbol{\mathrm{J}}_x \right| = \lambda^3 \left(\lambda^2-\pd{E^3_{111}}{E^1_1}\right) \left(\lambda^2-
    \pd{E^3_{112}}{E^1_2}\right) \left(\lambda^2-\pd{E^3_{113}}{E^1_3}
  \right),
\end{equation}
thus zero is a multiple eigenvalue of $\boldsymbol{\mathrm{J}}_x$. The
corresponding eigenvectors are
\[
(0,0,0,0,0,0,0,0,1)^T,\quad
(0,0,0,0,0,0,0,1,0)^T,\quad
(1,0,0,0,0,0,0,0,0)^T.
\]
In the case that
\[
\pd{E^3_{111}}{E^1_1} \geq 0, \quad 
\pd{E^3_{112}}{E^1_2} \geq 0, \quad
\pd{E^3_{113}}{E^1_3} \geq 0,
\]
the corresponding eigenvalues of the matrix are
\[
\lambda_1^{\pm} = \pm \sqrt{\pd{E^3_{111}}{E^1_1}}, \quad
\lambda_2^{\pm} = \pm \sqrt{\pd{E^3_{112}}{E^1_2}}, \quad
\lambda_3^{\pm} = \pm \sqrt{\pd{E^3_{113}}{E^1_3}}, \quad
\]
and the corresponding eigenvectors are
\[
  \begin{array}{cccccc}
    \left(\begin{array}{c}
        1 \\
        \lambda_1^{\pm} \\
        0 \\
        0 \\
        |\lambda_1^{\pm}|^2 \\
        0 \\
        0 \\
        \pd{E^3_{122}}{E^1_1} \\
        0
    \end{array} \right), \qquad
    \left(\begin{array}{c}
        0 \\
        0 \\
        -1 \\
        0 \\
        0 \\
        \lambda_2^{\pm} \\
        0 \\
        0 \\
        0
    \end{array}\right), \qquad
    \left(\begin{array}{c}
        0 \\
        0 \\
        0 \\
        -1 \\
        0 \\
        0 \\
        \lambda_3^{\pm} \\
        0 \\
        0
    \end{array}\right).
  \end{array}
\]
It could be verified directly that if any of the eigenvalues
$\lambda_i^{\pm}$, $i = 1,2,3$, equals zero, the Jacobian matrix is
not real diagonalizable. If we have
\begin{equation}\label{eq:cond-jacobi}
\pd{E^3_{111}}{E^1_1} > 0, \quad 
\pd{E^3_{112}}{E^1_2} > 0, \quad
\pd{E^3_{113}}{E^1_3} > 0,
\end{equation}
by the linear independence of the eigenvectors, one concludes that the
Jacobian matrix is real diagonalizable. Then the proof is finished by Lemma
\ref{lem:E3_11k-bound}.
\end{proof}
As a direct consequence of Theorem \ref{lem:Jx-invertible}, 
the 3D $B_2$ model is hyperbolic at equilibrium. This can be proved by the following
arguments. Let $\bR_j$, $j = 1,2,3$ be the three eigenvectors of $\bbrE^2$. Denote
the $k$-th component of the vector $\bR_j$ to be $R_{kj}$. Define the
Jacobian matrix of the 3D $B_2$ model \eqref{eq:moment_model} along $\bR_j$, $j = 1,2,3$
to be $\sum\limits_{k = 1}^3 R_{kj} \bJ_k$. Theorem \ref{lem:Jx-invertible}
shows that for the cases $\bE^1 = \boldsymbol{0}$, condition \eqref{eq:cond-sigma-positive} 
is the necessary and sufficient condition for the Jacobian matrix along $\bR_j$,
$\forall j = 1,2,3$ to be real diagonalizable. The above result holds
because for any given $\bR_j$, $j = 1,2,3$, we could always rotate the
coordinate system, such that $\bR_j$ is aligned with the
$x$-axis. Theorem \ref{thm:condition-nonnegative-ansatz}
gives \eqref{eq:cond-sigma-positive-1}
as the necessary and sufficient condition for $\sigma_1 > 0$, and rotation of coordinates can permute
the indices in \eqref{eq:cond-sigma-positive-1}, which results in \eqref{eq:cond-sigma-positive}.
Notice that at equilibrium, $\bbrE^2$ is a scalar matrix, so any direction is an eigenvector of
$\bbrE^2$. Therefore, the 3D $B_2$ model is hyperbolic at equilibrium. 

For given moments, we could always choose a
coordinate system such that $\bbrE^2$ is a diagonal matrix. The
system is hyperbolic if and only if for an arbitrary
$\bn \not= \boldsymbol{0}$, we always have
$n_x \bJ_x + n_y \bJ_y + n_z \bJ_z$ to be real diagonalizable.
For $\bE^1 = \boldsymbol{0}$, we sample for all possible
$(\lambda_1, \lambda_2, \lambda_3)$ and all unit vectors $\bn$, to
check if the matrix is real diagonalizable.
There is a hyperbolicity region around equilibrium for $\bE^1 =
\boldsymbol{0}$ as in Figure \ref{fig:hyperbolic_region}. 
The hyperbolicity region is a smaller
region than that enclosed by \eqref{eq:cond-sigma-positive}. However, it does cover a
neighborhood of the equilibrium.
\begin{figure}  
  \centering
  \includegraphics[width=0.48\textwidth]{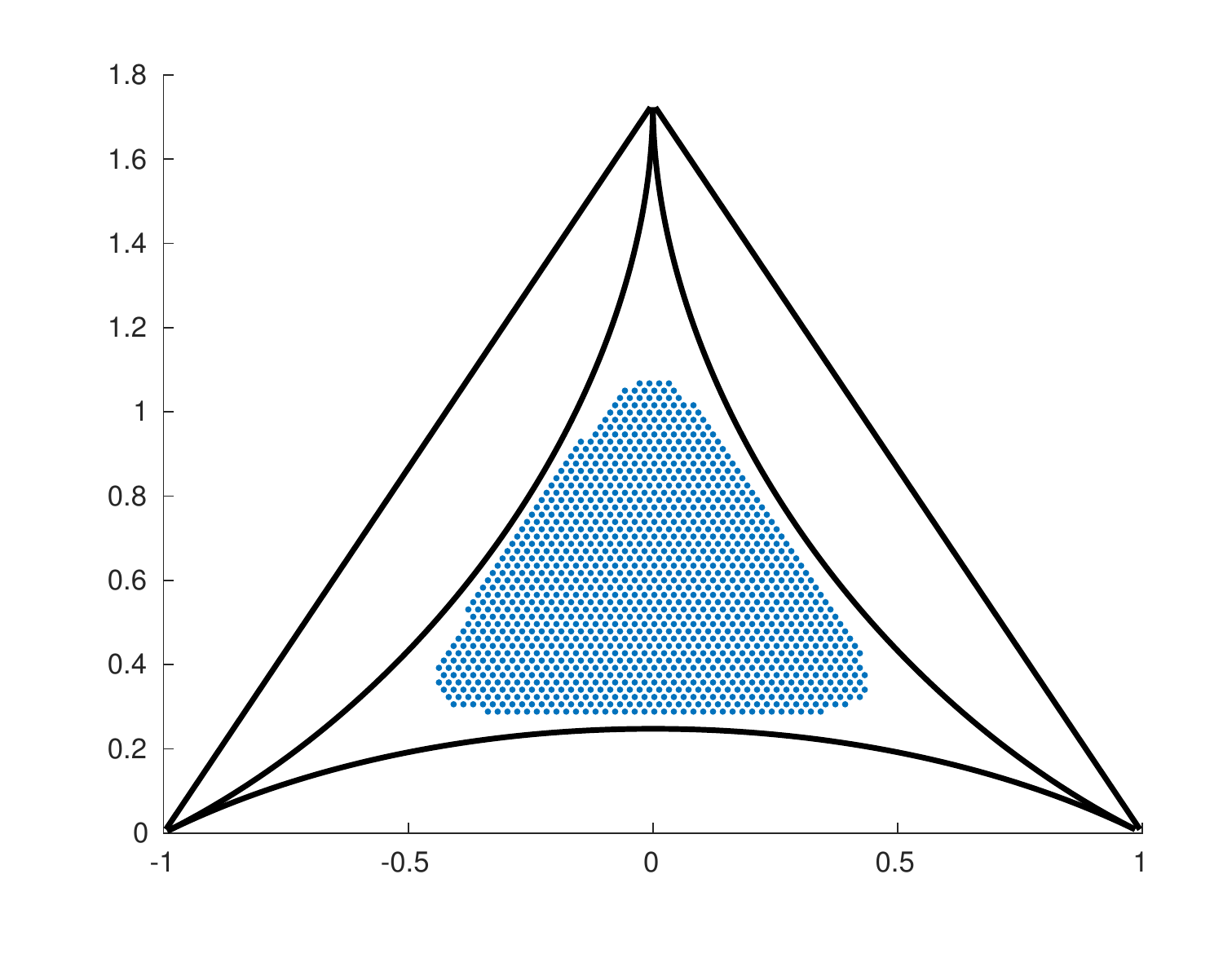}
  \caption{Region of hyperbolicity when $\bE^1 = \boldsymbol{0}$.
    $\left(\dfrac{\lambda_1}{E^0},\dfrac{\lambda_2}{E^0},
    \dfrac{\lambda_3}{E^0}\right)$ are taken as barycentric coordinates within the 
    triangle. The outer triangle is the realizability domain. The curves
    correspond to the outer boundary of the constraints \eqref{eq:cond-sigma-positive}.
    The 3D $B_2$ model is found to be hyperbolic within the dotted blue region. 
  }
  \label{fig:hyperbolic_region}
\end{figure}

Finally, we point out that although the 3D $B_2$ model is aimed at 
approximating the $M_2$ model, there is an interesting difference 
between them. This difference arises from the fact that
the ansatz is assumed to be the form $\bansatz$ in \eqref{eq:B2-ansatz}, 
and is independent of choice for the function $f(\mu; \gamma, \delta)$. 
When the given moments satisfy
\begin{equation} \label{eq:moments-with-axisymmetric-M2-ansatz}
  \exists i\not=j, \text{ such that } \lambda_i = \lambda_j,
  ~~\text{and}~~F_i = F_j = 0,
\end{equation}
the corresponding ansatz in the $M_2$ model is an axisymmetric
function. This includes the equilibrium distribution. Exactly at the
equilibrium, the 3D $B_2$ ansatz $\bansatz$ is isotropic, and, thus,
axisymmetric. However, even in neighbourhoods of the
equilibrium, moments corresponding to an axisymmetric ansatz in the
$M_2$ model would usually not reproduce an axisymmetric ansatz for the 3D
$B_2$ model. In other words, for arbitrary $\epsilon > 0$, there exist 
moments in the set
\[
\mathcal{A}_\epsilon = \left\{ (E^0,\bE^1,\bbrE^2) \in \cM ~\Big|~
  \bE^1 = \boldsymbol{0}, \sum_{i=1}^3 \left| \lambda_i -
    \dfrac{E^0}{3} \right|^2 < \epsilon, \text{ and
    \eqref{eq:moments-with-axisymmetric-M2-ansatz} is valid} \right\},
\]
for which the 3D $B_2$ ansatz $\bansatz$ is not axisymmetric; otherwise, 
the closure relation may lose the necessary regularities. More precisely, 
we claim:
\begin{theorem} 
  There are no functions
  $w_i(\lambda_1, \lambda_2, \lambda_3; F_1, F_2, F_3)$, $i=1,2,3$, in
  the 3D $B_2$ ansatz $\bansatz$ satisfying both items below:
  \begin{enumerate}
  \item $w_i$, $i=1,2,3$, are differentiable at the equilibrium state.
  \item The ansatz $\bansatz$ is axisymmetric for any moments in
    $\mathcal{A}_\epsilon$.
  \end{enumerate}
\end{theorem}
\begin{proof}
  We prove by contradiction. Suppose that \eqref{eq:B2-ansatz} is an
  axisymmetric distribution. Without losing generality we assume the
  corresponding moments satisfy $\lambda_2 = \lambda_3$, therefore the
  symmetric axis is aligned to $\bR_1$, and $F_2 = F_3 = 0$. To get
  axisymmetry in \eqref{eq:B2-ansatz}, the contributions from
  $w_2 f(\bOmega\cdot\bR_2;\gamma_2,\delta_2)$ and
  $w_3 f(\bOmega\cdot\bR_3;\gamma_3,\delta_3)$ have to be either zero
  or constant functions, hence $\sigma_2 = \dfrac{w_2}{3}$ and
  $\sigma_3 = \dfrac{w_3}{3}$, giving
  \[
  \sigma_1 + \sigma_2 + \sigma_3 = \lambda_1.
  \]
  Similar relations could be obtained when the symmetric 
  axis is aligned to $\bR_2$ or $\bR_3$.
  Consider the case when $\bE^1 = \boldsymbol{0}$. 
  Let 
  \[
  \sigma(\lambda_1,\lambda_2,\lambda_3) = 
  \sigma_1(\lambda_1,\lambda_2,\lambda_3;0,0,0)
  + \sigma_2(\lambda_1,\lambda_2,\lambda_3;0,0,0)
  + \sigma_3(\lambda_1,\lambda_2,\lambda_3;0,0,0).
  \]
  Based on the above arguments, we have
  \begin{equation}\label{eq:h-along-axis}
    \renewcommand\arraystretch{1.5}
    \sigma(\lambda_1,\lambda_2,\lambda_3) = \left \{
      \begin{array}{l}
        \lambda_1,\quad\text{if}\quad \lambda_2 = \lambda_3 
        = \frac12(E^0-\lambda_1),\\
        \lambda_2,\quad\text{if}\quad \lambda_1 = \lambda_3
        = \frac12(E^0-\lambda_2),\\
        \lambda_3,\quad\text{if}\quad \lambda_1 = \lambda_2
        = \frac12(E^0-\lambda_3).
      \end{array}\right.
  \end{equation}
  If all $w_i$, $i=1,2,3$, are differentiable, then all
  $\sigma_i$, $i=1,2,3$, are differentiable, so $\nabla \sigma$
  should be a continuous function for all realizabile
  moments. Let
  \[
  \bn_1 = (1, -\frac12, -\frac12),\quad
  \bn_2 = (-\frac12, 1, -\frac12),\quad
  \bn_3 = (-\frac12, -\frac12, 1),
  \]
  then $\nabla \sigma \cdot(\bn_1 + \bn_2 + \bn_3) = 0$. On the other
  hand, $\nabla \sigma \cdot\bn_1$ is equivalent to taking the
  derivative of $\sigma$ along
  $\lambda_2 = \lambda_3 = \frac12(E^0-\lambda_1)$, and we have
  similar relationships for $\bn_2$ and $\bn_3$.  So evaluating
  $\nabla \sigma \cdot\bn_j$ at $\lambda_j = \dfrac{E^0}{3}$,
  $j=1,2,3$ and $\bE^1 = \boldsymbol{0}$ according to
  \eqref{eq:h-along-axis} gives
  $\nabla \sigma \cdot(\bn_1 + \bn_2 + \bn_3) = 3$, leading to a
  contradiction. Therefore the two items can not be satisfied
  simultaneously.
\end{proof}
Notice that the proof of this lemma does not make use of the
specific form of the function $f$ in \eqref{eq:B2-ansatz}. In fact,
it can be seen from the proof that this inconsistency is due to the 
fact that the ansatz is a linear combination of three axisymmetric 
distributions. However, although the new model does not reproduce an 
axisymmetric ansatz for moments corresponding
to an axisymmetric ansatz in the $M_2$ model,
in such cases the closure of the new model retain the same structure 
as the $M_2$ closure. Without loss of generality consider the case when
$\lambda_2 = \lambda_3 = 0$ and $F_2 = F_3 = 0$. From
\eqref{eq:B2-closure}, we have
\begin{equation}\label{eq:axisymmetric-structure}
  \begin{split}
    & \Vint{ (\bOmega\cdot\bR_1)^2 (\bOmega\cdot\bR_2) \bansatz } =
    \Vint{ (\bOmega\cdot\bR_1)^2 (\bOmega\cdot\bR_3) \bansatz } =
    0,\\
    & \Vint{ (\bOmega\cdot\bR_1) (\bOmega\cdot\bR_2)^2 \bansatz } =
    \Vint{ (\bOmega\cdot\bR_3) (\bOmega\cdot\bR_2)^2 \bansatz }
    = \dfrac12\left(F_1-\Vint{ (\bOmega\cdot\bR_1)^3 }\right),
  \end{split}
\end{equation}
satisfying the same equalities as that given by an $M_2$ ansatz with
$\bR_1$ as the symmetric axis. 
 
Define
\[
  E_1 = \dfrac{\|\bE^1\|}{E^0},\quad
  E_2 = \dfrac{1}{(E^0)^3} (\bE^1)^T \bbrE^2 \bE^1, \quad
  E_3 = \dfrac{1}{(E^0)^4} \Vint{ \left(\bOmega\cdot\bE^1\right)^3 
  \hat{I}}.
\]
Then $E_i$, $i = 1,2,3$ would be the scaled first, second and
third-order moments for slab geometry cases. We compare the contour of
$E_3$ between the 3D $B_2$ model and the $M_2$ model for slab
geometry\footnote{The figure for the slab geometry was reproduced
  based on the data used to plot the corresponding figure in
  \cite{alldredge2016approximating}, and the computation was carried
  out by Dr. Alldredge using his own code during our collaboration
  therein.} in Figure \ref{fig:compare-contour}.  It is shown in
Figure \ref{fig:compare-contour} that the 3D $B_2$ model provides
realizable closure which is qualitatively similar to that of $M_2$
closure for most of realizable moments.
\begin{figure}
  \subfigure[The value of $E_3$  in slab geometry for 
  normalized realizable moments using
  the 3D $B_2$ closure.]{
    \label{fig:imageB2} 
    \includegraphics[width=0.48\textwidth]{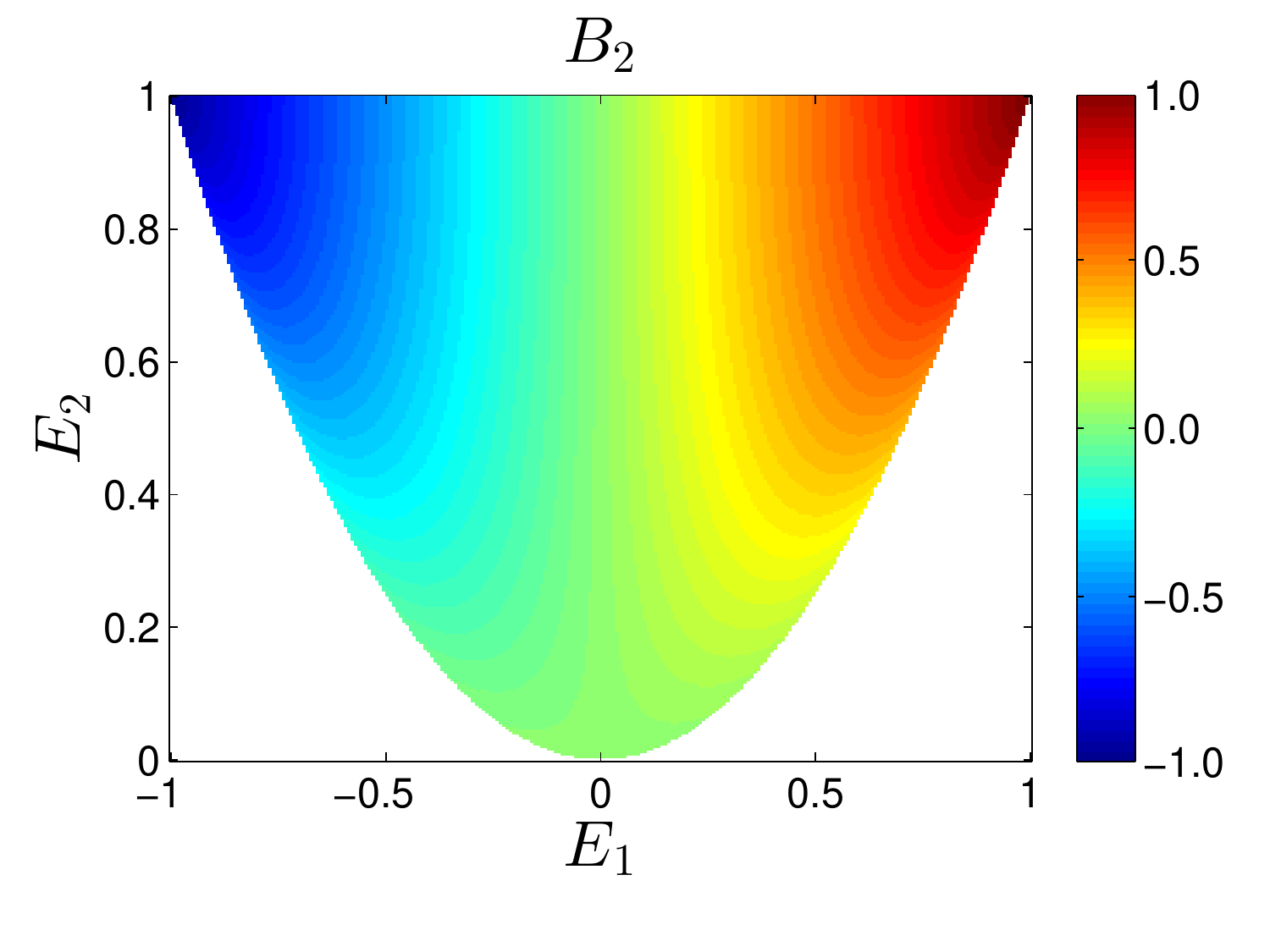}} \hfill
  \subfigure[The value of $E_3$ for normalized realizable moments
  using the maximum entropy closure in slab geometry for the
  monochromatic case.]{
    \label{fig:imageM2mono} 
    \includegraphics[width=0.48\textwidth]{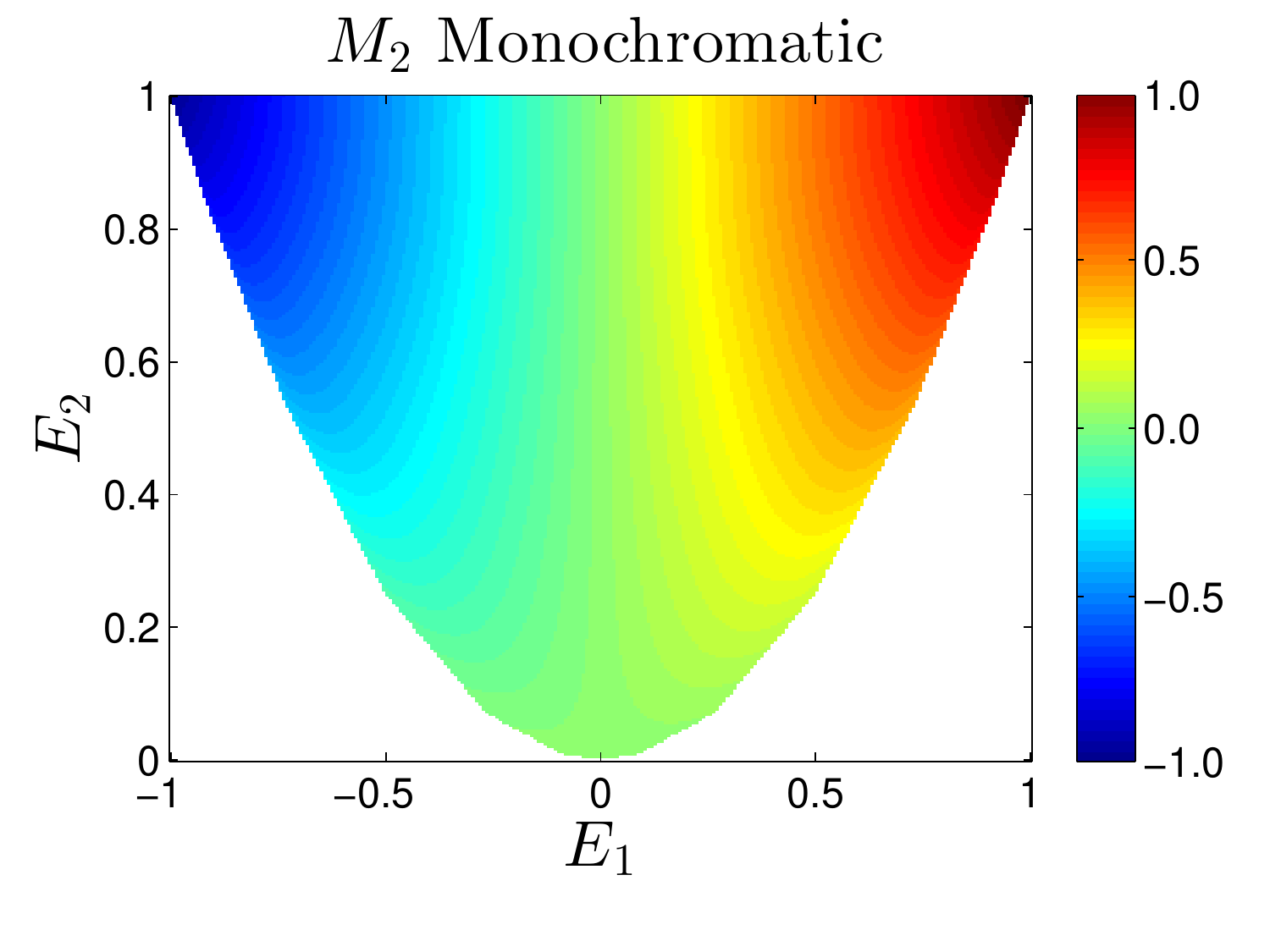}}
  \hfill
  \caption{Comparing the value of the closing moment $E_3$ for $M_2$
    for the monochromatic case and for the 3D $B_2$ model in slab
    geometry.}
  \label{fig:compare-contour}
\end{figure}


%% file: conclusion.tex

\section{Conclusion and Future Work}\label{sec:conclude}

We proposed a 3D $B_2$ model that is an extension of the EQMOM ($B_2$
model) in 1D. We showed, step by step, how the structure of the new
model is gradually refined. Particularly, we studied the main
properties of this new model, including rotational invariance,
realizability, and hyperbolicity.

We are currently carrying out numerical simulations using the new
model. For the first step, we hope that the model provides 
satisfactory results on standard benchmark problems.


%% file: acknowledgements.tex
\section*{Acknowledgements}
The authors appreciate the financial supports provided by \emph{the
  National Natural Science Foundation of China (NSFC)} (Grant No.
91330205, 11421110001, 11421101 and 11325102) and by \emph{the
  Specialized Research Fund for State Key Laboratories}.
